\documentclass[11pt]{amsart}
\usepackage[paperwidth=7in, paperheight=10in, margin=.875in]{geometry}
 \usepackage[backref,colorlinks,linkcolor=red,anchorcolor=green,citecolor=blue]{hyperref}
\usepackage{amsfonts,amssymb}
\usepackage{amsmath,bbm}
\usepackage{mathrsfs} 
\usepackage{graphicx}
\usepackage{cite}
\usepackage{enumerate}
\usepackage{tikz}
\usepackage{caption}
\usepackage{easyReview}
\usepackage{dsfont,mathtools}
\usepackage[ruled,vlined]{algorithm2e}

\numberwithin{equation}{section}
\newtheorem{thm}{Theorem}[section]

\newtheorem{prop}[thm]{Proposition}

\newtheorem{lem}[thm]{Lemma}

\theoremstyle{remark}

\theoremstyle{definition}

\newcommand{\suchthat}{\;\ifnum\currentgrouptype=16 \middle\fi|\;}

\newcommand{\Gnorm}[1]{\left\lVert#1\right\rVert}

\newtheorem{rem}[thm]{Remark}

\makeatletter
\@namedef{subjclaame@2020}{\textup{2020}}
\makeatother
\allowdisplaybreaks
\begin{document}
 \title[]{Orthogonal gamma-based expansion for the CIR's first passage time distribution} 
\author[E. Di Nardo]{Elvira Di Nardo$^{\ast}$}
\address{$^{\ast}$ Dipartimento di Matematica \lq\lq G. Peano\rq\rq, Università degli Studi di Torino, Via Carlo Alberto 10, 10123 Torino, Italy}
\email{elvira.dinardo@unito.it }
\author[G. D'Onofrio]{Giuseppe D'Onofrio$^{\dagger}$}
\address{$^{\dagger}$ Dipartimento di Scienze Matematiche, Politecnico di Torino,  10129 Torino, Italy}
\email{giuseppe.donofrio@polito.it}
\author[T. Martini]{Tommaso Martini$^{\ast\ast}$}
\address{$^{\ast\ast}$ Dipartimento di Matematica \lq\lq G. Peano\rq\rq, Università degli Studi di Torino, Via Carlo Alberto 10, 10123 Torino, Italy}
\email{tommaso.martini@unito.it }            \pagestyle{myheadings}  \maketitle
\begin{abstract}
In this paper we analyze a method for approximating the first-passage time density and the corresponding distribution function for a CIR process. This approximation is obtained by truncating a series expansion involving the generalized Laguerre polynomials and the gamma probability density.  The suggested approach involves a number of numerical issues  which depend strongly on the  coefficient of variation of the first passage time random variable. These issues are examined and solutions are proposed also involving the first passage time  distribution function. Numerical results and comparisons
with alternative approximation methods  show the strengths and
weaknesses of the proposed method. A general acceptance-rejection-like procedure, that makes use of the approximation, is presented. It allows the generation of first passage time data, even if  its distribution is
unknown.
\end{abstract}
\noindent {\small {\bf keywords: }{Feller square-root process, hitting times, Fourier series expansion, cumulants, Laguerre polynomials, acceptance-rejection method }} 

\section{\label{sec:level1}
Introduction }

In many applications spanning from finance to engineering including, among others, computational neuroscience, mathematical biology and reliability theory (see \cite{redner} for a thorough exposition) the dynamics of a noisy system is described by a stochastic process $Y(t)$ evolving in the presence of a threshold $S(t).$
The first-passage-time (FPT) problem consists in finding the distribution of the random variable (rv) 
$T$, defined by
\begin{equation}\label{12}
T=\begin{cases} \inf_{t\geq \tau}\{Y(t)>S(t)\},& Y(\tau)=y_{\tau}<S(\tau),\\
\inf_{t\geq \tau}\{Y(t)<S(t)\},& Y(\tau)=y_{\tau}>S(\tau),\\
\end{cases}
\end{equation}
representing the time the process $Y(t)$ crosses the threshold $S(t)$ for the first time.
Although classical and very easy to state, the solution in closed form of this problem is available only in a very few cases, depending on the properties of both $Y(t)$ and $S(t).$ In this paper, we address the FPT problem of a Cox-Ingersoll-Ross (CIR) process $Y(t)$ through a constant threshold $S$. This one-dimensional diffusion process, belonging to the class of Pearson diffusions \cite{forman2008}, is frequently involved in the field of mathematical finance starting from the seminal paper \cite{CIR} by which it is commonly called nowadays. Outside this community the process is often called square-root, due to the form of its diffusion coefficient (or volatility), or, due to historical reasons, Feller process from the $1951$ paper, in which the process is introduced for the first time \cite{fel51}. 

When considering the FPT problem of a CIR process, the literature, also quite recent, is vast but the results are partial and fragmentary,
see for instance
\cite{ascione2023ergodicity}, \cite{deaconu2013hitting}, 
\cite{gerhold2020running}, \cite{giorno1988}, \cite{giorno2021first},  \cite{going2003survey}, \cite{Linetsky}, \cite{martin2011first}, \cite{song2016first}, and \cite{masoliver2012first} for a thorough review of the state of art.
Exploiting the Laplace transform of the FPT pdf and the theory of formal power series, a closed form expression of its FPT cumulants is given in \cite{di2021cumulant}. In the same paper, under appropriate assumptions, the FPT pdf has been expanded in series  of generalized Laguerre polynomials, involving moments computed from cumulants and weighted by a gamma pdf. The idea of approximating a pdf by truncating a suitable series expansion is not new. Indeed 
such an approximation is of the Gram-Charlier type with a gamma distribution rather than a normal distribution as reference (parent) distribution and with generalized Laguerre polynomials instead of Hermite polynomials as multipliers. 
In \cite{provost2016distribution} 
a general methodology to approximate a pdf based on the knowledge of its moments is introduced, using the product
of a suitable weight function, as parent distribution, and a suitable family of associated orthogonal polynomials, as multipliers. 
One of the main issues of this approximation is that negative values can occur, although the approximate density always entails a unit area. This happens even with the Gram-Charlier series having Gaussian parent distribution. To overcome this drawback, two approaches can be found in the literature. The first one is to use the approximation with a low truncation order and to find constrained regions on the values
of the cumulants (or moments) that admit a valid (non-negative) pdf. The suggested truncation is mostly at the fourth-order term because it becomes difficult to manage  valid regions for higher orders. Within the FPT framework, this approach was used to approximate the FPT pdf of an Ornstein-Uhlenbeck process 
\cite{smith1991laguerre}. In 
particular, the restrictions on the first four moments that guarantee the non-negativity of the approximated density    are outlined and discussed in \cite{lung1998approximations}, along with a thorough examination of when to apply this approximation. Indeed, with this low truncation order, the approximated density may fail to be close to the theoretical one, especially for distributions that are not sufficiently close to
the parent distribution. In 
the case of Gaussian parent distribution and for arbitrary even order, the valid region of cumulants has been found numerically through a semi-definite algorithm 
\cite{lin2022valid}.  A second way of tackling this issue consists in replacing values of a suitable positive interpolating function to the negative ones assumed by the approximated pdf. In \cite{Wil73}, as interpolating function for pdfs with support $(0,\infty),$ a second-degree polynomial is suggested in a right-handed neighborhood of the origin. In this paper, taking into account that  the FPT pdfs are unimodal for diffusion processes \cite{Uwe80}, this second approach is 
developed along a different direction and for the first time within the FPT framework.
Firstly, sufficient conditions are given on the sign of the coefficients of the series expansion so that the approximated density may hold non-negative values on the tails and in a right-handed neighborhood of the origin. 
Secondly, if there are additional intervals in which the approximated pdf turns out to be negative, an appropriate correction is proposed that takes into account  sufficient conditions given in \cite{di2021cumulant} allowing
the Laguerre-Gamma type expansion for the FPT pdf. The issue concerning the possible negative values of the approximated pdf can be overcome also by considering the FPT cumulative distribution function (cdf). For this reason, in parallel with our discussion, we develop the method for the approximation of the cdf as well. We stress that this approach, new in the FPT context, has some numerical advantages
and allows an easier approximation of quantiles. 

Those highlighted so far are not the only issues concerning the use of such an approximation. The choice of the gamma pdf parameters as well as the order of truncation of the series are additional issues that may affect the quality of the approximation. These issues, only briefly sketched in \cite{provost2016distribution}, are considered in detail in this paper. For example, the key role played by the coefficient of variation in the choice of the gamma pdf parameters is shown, while the truncation order is controlled by appropriate stopping criteria. 

The proposed Laguerre-Gamma expansion has an additional advantage since 
density estimates can be produced based on sample moments. Indeed, if the
FPT moments/cumulants are not known, this approach allows 
to recover an approximation of the FPT pdf starting from a sample of FPT data.  
These estimators are known in the literature as orthogonal series estimators and can be very competitive when compared with the classical density estimators such as the kernel density estimator (KDE) or the histogram \cite{hall1980estimating}.

Thanks to obtained evaluations of the approximation error, an acceptance-rejection-like method,  that makes use of the series expansion, is finally proposed. 
It allows the generation of FPT data, even if its distribution is unknown, and can be applied to a wide class of pdfs. Note that, 
although never used for the CIR process, acceptance-rejection methods had already appeared in the FPT context (see for instance \cite{herrmann2020exact} 
and \cite{mijatovic2015randomisation}), but the approximation strategy here proposed is a novelty.
This method is particularly useful since exact simulation techniques for CIR sample paths are not available, and the existing ones, based on discretization methods or transition densities, exhibit  large computational costs if the fixed time step is small.  

The paper is organized as follows. 
In Section \ref{section2} we resume the FPT problem for the CIR process recalling the known results useful for carrying out the proposed approximation.
In Section \ref{section3} we discuss
the convergence of the method. Moreover we address some theoretical issues closely related to the approximation such as the choice of the truncation order.
The role played by the coefficient of variation of the FPT rv in the choice of the gamma pdf parameters is also discussed.  Section \ref{section4} suggests how to overcome the two main computational issues arising in dealing with such an approximation: the monotonicity of the approximated cdf and the positivity of the approximated pdf. Numerical results and comparisons with alternative approximation methods 
are given in Section \ref{section5} aiming to discuss  the strengths and weaknesses of the proposed approach. We  set three different choices of the CIR process parameters and boundaries that corresponds to different forms and statistical properties of the FPT pdf.  An application of the Laguerre-Gamma approximation is shown in the last section, which involves sampling FPTs using a technique analogous to the acceptance-rejection method. Concluding remarks close the paper.
\section{The CIR process and the FPT problem} \label{section2}
The CIR process we refer to is the unique strong solution of the stochastic differential equation \cite{fel51} 
\begin{equation}
\label{eqn: defn feller}
dY(t) = (-\tau Y(t) + \mu) \, {\rm d}t + \sigma \sqrt{Y(t) - c}\,  {\rm d}W(t),
\end{equation}
where $W(t)$ is a standard Brownian motion, $c\le 0$, $\tau > 0$, $\mu \in \mathbb{R}$, $\sigma > 0$ and $Y_0 = y_0$. The state space of the process is the interval $(c,+\infty)$. The endpoints $c$ and $+\infty$ can or cannot be reached in a finite time, depending on the underlying parameters. According to the Feller classification of boundaries \cite{karlin1981second}, $c$ is an entrance boundary if it cannot be reached by $Y(t)$ in finite time, and there is no probability flow to the outside of the interval $(c,+\infty)$.
In particular,
$$ c \,\,\, \hbox{\rm is an entrance boundary if}
\,\,\, s :=2(\mu-c \tau) / \sigma^{2} \geq 1.$$
This will be a standing assumption in the following.
\par
Denote with $g(t)=\frac{d}{d t}\mathbb{P}\{T \leq t\}$ the pdf of the FPT rv $T$ as defined in (\ref{12}). Its Laplace transform $\widetilde{g}(z)$ is such that
$\widetilde{g}(z)=1$ if $y_{0} \equiv S$ and $\widetilde{g}(z)<+\infty$ for any different $y_{0}$
\cite{masoliver2012first}.
Its closed form expression is \cite{d2018two}
\begin{equation}
\widetilde{g}(z)=\frac{\Phi\left(\frac{z}{\tau}, s, \frac{2 \tau\left(y_{0}-c\right)}{\sigma^{2}}\right)}{\Phi\left(\frac{z}{\tau}, s, \frac{2 \tau(S-c)}{\sigma^{2}}\right)}, \quad z>0
\label{(LT)}
\end{equation}
where $\Phi(a, b, z)={ }_{1} F_{1}(a ; b ; z)$ is the confluent hypergeometric function of the first kind (or Kummer's function). The Laplace transform \eqref{(LT)} cannot be inverted explicitly, except for the case $S=0$, see for instance \cite{martin2011first}, but information on the moments can be obtained by direct derivation or from cumulants as described in the next subsection.

\subsection{FPT cumulants and moments} 
Recall that, if $T$ has moment generating function  $\mathbb E[e^{zT}] < \infty$ for all $z$ in an open interval about $0,$ then its cumulants $\{c_k(T)\}_{k \geq 1}$ are such that 
\begin{equation*}
\sum_{k \geq 1} c_k(T) 
\frac{z^k}{k!}  = \log \mathbb E[e^{zT}] 
\label{defcum}
\end{equation*}
for all $z$ in some (possibly smaller) open interval about $0.$ 
Using the logarithmic polynomials\footnote{See Appendix for their definition.} $\{P_{k}\}$, the FPT cumulants of the 
CIR process can be expressed as  \cite{di2021cumulant}  
$$c_{k}(T)=(-\tau)^{-k}\left[c_{k}^{*}\left(y_{0}\right)-c_{k}^{*}(S)\right], \,\,\,k \geq 1 $$ 
where
\begin{equation}
c_{k}^{*}(w)=P_{k}\left[h_{1}\left(\frac{2 \tau(w-c)}{\sigma^{2}}\right), h_{2}\left(\frac{2 \tau(w-c)}{\sigma^{2}}\right), \ldots, h_{k}\left(\frac{2 \tau(w-c)}{\sigma^{2}}\right)\right],
\label{cumulants}
\end{equation}
with $
h_{j}(y)=j ! \sum_{n \geq j}\left[\begin{array}{c}
n \\
j
\end{array}\right] \frac{y^{n}}{n !\langle s\rangle_{n}},$ for $j=1,2, \ldots, k,$ $\left[\begin{array}{l}n \\ j\end{array}\right]$ the unsigned Stirling numbers of first type and $\langle\cdot\rangle_{n}$ the $n$-th rising factorial.

FPT moments of the CIR process are obtained from cumulants \cite{di2021cumulant} using the complete Bell polynomials\footnote{See Appendix for their definition.} $\left\{Y_{k}\right\}$ and $\{c^*_k\}$ given in \eqref{cumulants}, that is
\begin{equation}
\mathbb{E}\left[T^{k}\right]=\frac{(-1)^{k}}{\tau^{k}} \sum_{i=0}^{k}\left(\begin{array}{c}
k \\
i
\end{array}\right) Y_{k-i}\left[c_{1}^{*}\left(y_{0}\right), \ldots, c_{k-i}^{*}\left(y_{0}\right)\right] Y_{i}\left[-c_{1}^{*}(S), \ldots,-c_{i}^{*}(S)\right]
\label{(momT)}
\end{equation}
for $k \geq 1.$ An alternative way to compute moments from cumulants is the well-known recursion formula \cite{DiNardo2006}
\begin{equation}
\mathbb{E}\left[T^{k}\right]=c_{k}(T)+\sum_{i=1}^{k-1}\left(\begin{array}{c}
k-1 \\
i-1
\end{array}\right) c_{i}(T) \mathbb{E}\left[T^{k-i}\right].
\label{eq:recursionmoments}
\end{equation}
This formula is particularly convenient from a computational point of view and has been used to recover FPT moments from the knowledge of cumulants.

\subsection{The FPT pdf and cdf}\label{section3}
Under suitable hypotheses, a closed form expression of the FPT pdf has been given in \cite{di2021cumulant} using the moments $\{\mathbb{E}\left[T^{k}\right]\}.$ Indeed, suppose
\begin{equation}
\label{perprop31}
f_{\alpha,\beta}(t) = \beta (\beta t)^{\alpha} \frac{e^{-\beta t}}{ \Gamma(\alpha+1)}, \,\, t > 0
\end{equation}
the gamma pdf with scale parameter $\alpha+1>0$ and shape parameter $\beta>0.$  For $\alpha > -1,$ let the polynomial sequence
$\{Q_k^{(\alpha)}(t)\}_{k \geq 0}$ be defined as
\begin{equation}
Q_k^{(\alpha)}(t) = (-1)^k \left( \frac{\Gamma(\alpha+1+k)}{k! \, \Gamma(\alpha+1)}\right)^{-1/2} \!\!\! {L}_k^{(\alpha)}(t),
\label{(orthonormale)}
\end{equation}
where ${L}_k^{(\alpha)}(t)$ is the $k$-th generalized Laguerre polynomial 
\begin{equation*}
L_k^{(\alpha)}(t) = \sum_{i=0}^k 
\binom{k+\alpha}{k-i} \frac{(-t)^i}{i!}, \,\,\,\, k \geq 1
\end{equation*}
with $L_0^{(\alpha)}(t) = 1.$  For any fixed $t > 0,$ the FPT pdf admits the following expansion (see Theorem $2$ in \cite{di2021cumulant}):
\begin{equation}
g(t) = \frac{\beta (\beta t)^{\alpha} e^{- \beta t}}{\Gamma(\alpha+1)} \sum_{k \geq 0} a_k^{(\alpha)}{Q}_k^{(\alpha)}(\beta t), 
\label{approximationgen2}
\end{equation}
where $a_k^{(\alpha)}={\mathbb E}[Q_k^{(\alpha)}(\beta T)]$ for  $k \geq 0.$ 
\begin{rem}
{\rm If 
\begin{equation*}
\int_0^{\infty} t^{-\alpha} e^{\beta t} [g(t)]^2 < \infty \qquad  {\rm then}
\qquad \frac{g(t)}{f_{\alpha,\beta}(t)} \in {\mathcal L}^2(\nu),
\end{equation*}
where ${\mathcal L}^2(\nu)$ is the Hilbert space of the square-integrable functions with respect to the measure $\nu$ having density $f_{\alpha,\beta}(t).$ Therefore,  
\eqref{approximationgen2} represents the Fourier-Laguerre series expansion of $g(t)/f_{\alpha,\beta}(t)$ in terms of the complete orthonormal sequences $\{{Q}_k^{(\alpha)}(t)\}.$}
\end{rem}
Some algebra allows us to write the FPT pdf $g(t)$ in (\ref{approximationgen2}) as
\begin{equation}
g(t) = \beta (\beta t)^{\alpha} e^{- \beta t}  \sum_{k \geq 0} {\mathcal B}_k^{(\alpha)}{L}_k^{(\alpha)}(\beta t), \,\, t > 0
\label{approximationgen3}
\end{equation}
with coefficients ${\mathcal B}_0^{(\alpha)} = 1$ and
\begin{equation}
{\mathcal B}_k^{(\alpha)} = (-1)^k a_k^{(\alpha)} \left( \frac{\Gamma(\alpha+1) \Gamma(\alpha+1+k)}{k!} \right)^{-1/2} = \sum_{j=0}^k \binom{k}{j} \frac{(-\beta)^j {\mathbb E}(T^j)}{\Gamma(\alpha+j+1)},
\label{(coeffak1)}
\end{equation}
depending on the moments of $T.$ By recalling that
$$\int \tau^{\alpha} e^{-\tau} L_k^{(\alpha)}(\tau) \, {\rm d} \tau = \frac{\tau^{\alpha+1} \Gamma(\alpha + k + 1)}{k!} \frac{\Phi(\alpha+k+1 ; \alpha + 2 ; - \tau)}{\Gamma(\alpha + 2)}$$
where $\Phi(a, b, z)={ }_{1} F_{1}(a ; b ; z)$ is the confluent hypergeometric function of the first kind,
and using \eqref{approximationgen3}, a closed form expression of the FPT cdf is given in the following statement. 
\begin{prop}
The FPT cdf $G(t)$ is 
\begin{equation*}
    G(t)= \frac{(\beta t)^{\alpha+1}}{\Gamma(\alpha+2)} \sum_{k \geq 0}    \frac{\Gamma(\alpha+k+1)}{k!} \, {\mathcal B}_k^{(\alpha)} \, \Phi(\alpha+k+1, \alpha+2, -\beta t), \,\,\,\,\, t > 0.
\end{equation*}    
\end{prop} 
\section{\label{section3}
The FPT approximation}
An approximation of the FPT pdf can be recovered from \eqref{approximationgen2} by using a truncation of the series up to an order $n$
\begin{equation}
\hat{g}_n(t) = \frac{\beta (\beta t)^{\alpha} e^{- \beta t}}{\Gamma(\alpha+1)} \sum_{k = 0}^n a_k^{(\alpha)}{Q}_k^{(\alpha)}(\beta t), \,\,\,\, t > 0.
\label{approximationgen2bis}
\end{equation}
The higher is the order $n$ the better should be the approximation. Indeed 
the ${\mathcal L}^2(\nu)$-error in replacing $g(t)$ with its approximation $\hat{g}_n(t)$ given in \eqref{approximationgen2bis} is\cite{MR1563024} 
\begin{equation}
\label{errorPars}
\Gnorm{\frac{g-\hat{g}_n}{f_{\alpha,\beta}}}_{\alpha,\beta}  = \left[\sum_{k \geq n+1} \big(a_k^{(\alpha)}\big)^2 \right]^{1/2}
\end{equation}
where $\Gnorm{\,\,}_{\alpha,\beta}$ denotes 
the norm in ${\mathcal L}^2(\nu).$ 
Thus the error may be estimated by calculating the rate of decrease of $a_k^{(\alpha)}$
when $k \rightarrow \infty.$ The latter is given in the following proposition.
\begin{thm}\label{thm:a_k}
Assume the FPT pdf $g(t) \in C^2[0,+\infty).$  If
\begin{equation*} 
\beta < \frac{2}{\mathbb{E}[T]} \quad \hbox{\rm and} 
\quad g(t) = o(t^{\delta}) \quad \hbox{\rm for} \,\,\, t \rightarrow 0, \quad \hbox{\rm with} \,\, \, \delta > \frac{\alpha}{2}+1
\end{equation*}
then in \eqref{errorPars} $a_k^{(\alpha)}  = {\mathbb E}[Q_k^{(\alpha)}(\beta T)] =  O(k^{-1})$ as $k \rightarrow \infty.$
\end{thm} 
\begin{proof}
Observe that $a_k^{(\alpha)} ={\mathbb E}[Q_k^{(\alpha)}(\beta T)]$ gives 
\begin{equation}
a_k^{(\alpha)} = \frac{1}{\beta} \int_0^{\infty} Q_k^{(\alpha)}(t) \, \tilde{g}_{\alpha,\beta}(t) \, t^{\alpha} \, e^{-t} \, {\rm d}t \quad {\rm where} \quad \tilde{g}_{\alpha,\beta}(t) = \frac{g(t/\beta)}{t^{\alpha} \, e^{-t}}.
\label{akuno}
\end{equation}
As the generalized Laguerre polynomials $\{L_k^{(\alpha)}(t)\}$ are eigenfunctions of a Sturm-Liouville problem \cite{aminataeirational} with associated eigenvalues $\lambda_k = k$
$$ \frac{{\rm d}}{{\rm d}t} \left(t^{\alpha+1} e^{-t} y^{\prime} \right) + k \,t^{\alpha} \, e^{-t} \, y = 0 \quad {\rm with} \quad y=y(t), \, k \geq 1$$
the same happens for the linearly transformed polynomials $\{Q_k^{(\alpha)}(t)\}$ in \eqref{(orthonormale)}. Therefore in \eqref{akuno}, replace $$Q_k^{(\alpha)}(t) \, t^{\alpha} \, e^{-t} \quad {\rm with} \quad - \frac{1}{k}  \frac{{\rm d}}{{\rm d}t} \left(t^{\alpha+1} e^{-t} y^{\prime} \right).$$
Integrating by parts the integral in \eqref{akuno} and neglecting the constants, the rhs of \eqref{akuno} reads
\begin{equation}
a_k^{(\alpha)} \approx \frac{1}{k} \int_0^{\infty}
 t^{\alpha+1} \, e^{-t}
 \frac{{\rm d}}{{\rm d}t} [Q_k^{(\alpha)}(t)] 
 \frac{{\rm d}}{{\rm d}t}  [\tilde{g}_{\alpha,\beta}(t)] \, {\rm d}t.
\label{akdue}
\end{equation}
Indeed we have
\begin{equation}
\lim_{t \rightarrow 0} \,\, \tilde{g}_{\alpha,\beta}(t) \, t^{\alpha+1} \, e^{-t} \, \frac{{\rm d}}{{\rm d}t} [Q_k^{(\alpha)}(t)] = 0 \quad {\rm and} \quad \lim_{t \rightarrow \infty} \,\,\tilde{g}_{\alpha,\beta}(t) \, t^{\alpha+1} \, e^{-t} \, \frac{{\rm d}}{{\rm d}t} [Q_k^{(\alpha)}(t)] = 0.
\label{limituno}
\end{equation}
The first limit in \eqref{limituno}
results  by the hypothesis $g(t)=o(t^{\delta})$ for $t \rightarrow 0.$ The second limit in \eqref{limituno} follows by taking into account 
that the FPT pdf of one-dimensional diffusion processes with steady-state distribution is known to be approximately exponential  for $t \rightarrow \infty$, with parameter ${\mathbb E}[T]^{-1}$ 
\cite{nobile_ricciardi_sacerdote_1986}. 
Integrating by parts the integral in \eqref{akdue} and neglecting the constants, the rhs of \eqref{akdue} reads
\begin{equation}
a_k^{(\alpha)} \approx - \frac{1}{k} \int_0^{\infty}
Q_k^{(\alpha)}(t)  \frac{{\rm d}}{{\rm d}t}  \left(t^{\alpha+1} \, e^{-t}
 \frac{{\rm d}}{{\rm d}t}  [\tilde{g}_{\alpha,\beta}(t)] \right) \, {\rm d}t,
\label{aktre}
\end{equation}
where similar considerations done for \eqref{limituno} apply for recovering 
$$\lim_{t \rightarrow 0} \,\,  \, t^{\alpha+1} \, e^{-t} Q_k^{(\alpha)}(t)\, \frac{{\rm d}}{{\rm d}t} [\tilde{g}_{\alpha,\beta}(t)] = 0 \quad {\rm and} \quad \lim_{t \rightarrow \infty} \,\, t^{\alpha+1} \, e^{-t} Q_k^{(\alpha)}(t)\, \frac{{\rm d}}{{\rm d}t} [\tilde{g}_{\alpha,\beta}(t)] = 0.$$
Now, in \eqref{aktre} set
$$h(t) = \frac{1}{w(t)} \frac{{\rm d}}{{\rm d}t}  \left(t^{\alpha+1} \, e^{-t}
 \frac{{\rm d}}{{\rm d}t}  [\tilde{g}_{\alpha,\beta}(t)] \right) \quad {\rm with} \quad w(t)=t^{\alpha} e^{-t}.$$
Applying the Cauchy-Schwarz inequality to the rhs of \eqref{aktre}, we get
$$ \left| \int_0^{\infty}
\sqrt{w(t)} Q_k^{(\alpha)}(t)  h(t) \sqrt{w(t)} \, {\rm d}t \right|^2 \leq 
\left( \int_0^{\infty}
w(t) [Q_k^{(\alpha)}(t)]^2 
\, {\rm d}t \right) \left( \int_0^{\infty}
 [h(t)]^2 w(t) 
\, {\rm d}t  \right)$$
which is finite and not depending on the order $k,$ if the same is true for both integrals on the lhs. Observe that the first integral corresponds to the orthonormality condition of $Q_k^{(\alpha)}(t)$ and so it is finite and not depending on $k.$  The second integral does not depend on $k$ and is finite if the integrand is smooth and the limits for $t\rightarrow 0$
and $t \rightarrow \infty$ are finite. 
Note that
$$\frac{h(t)}{\sqrt{w(t)}} = \frac{(\alpha+1-t) \frac{{\rm d}}{{\rm d}t}[\tilde{g}_{\alpha,\beta}(t)]  + t \frac{{\rm d}^2}{{\rm d}t^2} [\tilde{g}_{\alpha,\beta}(t)]}{t^{-\alpha/2} \, e^{t/2}}.$$
Thus for $t \rightarrow \infty$
and $\beta < 2 /{\mathbb E}[T]$ we have
\begin{equation*}
\lim_{t \rightarrow \infty} \,\, \frac{h(t)}{\sqrt{w(t)}} =  \lim_{t \rightarrow \infty} 
\frac{e^{z[1-(\beta {\mathbb E}[t])^{-1}]}}{z^{\alpha/2} e^{z/2}} + \lim_{t \rightarrow \infty} \frac{e^{z[1-(\beta {\mathbb E}[t])^{-1}]}}{z^{\alpha/2 -1} e^{z/2}} = 0
\end{equation*}
assuming $g(t) = O( e^{-t/(\beta {\mathbb E}[T])}).$ 
Instead for $t \rightarrow 0$ the limit reduces to
\begin{equation*}
\lim_{t \rightarrow 0} \,\, \frac{h(t)}{\sqrt{w(t)}} = \lim_{t \rightarrow 0} 
e^{t/2} t^{\delta-\alpha/2-1} + \lim_{t \rightarrow 0} 
e^{t/2} t^{\delta-\alpha} = 0
\end{equation*}
if $\delta > 1 + \alpha/2.$
\end{proof}

The request of the existence of the second derivative in Theorem \ref{thm:a_k} is a reasonable assumption, since the property can be seen as a consequence of the following observation.
\begin{rem}
Following \cite{pauwels1987smooth},  we could have alternatively asked that there exists $\varepsilon>0$ such that $\sigma \sqrt{x} \geq \varepsilon$, for all  $x$ in the state space of the process. This condition implies, in this case, the existence and boundedness at least of the first two derivatives of the FPT pdf $g(t)$.
To investigate Pauwels' condition, one can follow Feller's classification of the boundaries \cite{feller1952parabolic}.  Using the transition densities of the Feller  process \cite{masoliver2012first}, one has to show that the flux through the value $\varepsilon$ is zero or that the capacity of the interval $[0,\varepsilon)$  vanishes \cite{bertini2008modelling}. For a fixed, small $\varepsilon>0$ we observe, at least numerically, that the capacity of the interval $[0,\varepsilon)$ (see formula 19 in \cite{bertini2008modelling})
goes to zero as $\mu$ increases. This means that for a \lq\lq large enough'' choice of $\mu$, the mentioned assumption in Theorem \ref{thm:a_k} is satisfied.
\end{rem}

By truncating the series
in \eqref{approximationgen3} up to the order $n$, the approximated $\hat{g}_n(t)$ in \eqref{approximationgen2bis} can be rewritten as 
\begin{equation}
\label{expansion3}
\hat{g}_n(t) = f_{\alpha,\beta}(t) p_n(t) \,\,\,\,\,\,\,\, {\rm where} \,\,\,\,\,\,  p_n(t)  = \sum_{k=0}^n h_{n,k} \frac{(-\beta t)^k}{k!} 
\end{equation}
with
\begin{equation*}
h_{n,k} = \sum_{j=k}^n {\mathcal B}_j^{(\alpha)}
\binom{\alpha+j}{j-k} \quad \hbox{\rm and} \quad \binom{\alpha+j}{j-k} =
\left\{ \begin{array}{ll}
1, & j=k,\\
\frac{(\alpha+j)(\alpha+j-1)\cdots (\alpha+k+1)}{(j-k)!},
& j > k.
\end{array}\right.
\end{equation*}
In particular $\{h_{n+1,i}\}_{i=0}^{n+1}$
can be recovered from 
$\{h_{n,i}\}_{i=0}^n$ 
using the following recursion formula \cite{di2023approximating}
\begin{equation}
\label{rech}
    h_{n+1,i}= \left\{ \begin{array}{ll}
{\mathcal B}_{n+1}^{(\alpha)}, & \text{for } i=n+1, \\
h_{n,i} + {\mathcal B}_{n+1}^{(\alpha)} \binom{\alpha+n+1}{n+1-i}, &  \text{for } i=0, \ldots, n.
\end{array} \right.
\end{equation}
Likewise, the coefficients $\{{\mathcal B}_{n+1}^{(\alpha)}\}$ can be recovered from $\{{{\mathcal B}_{j}}^{(\alpha)}\}_{j=0}^n$ as
\begin{equation}
{\mathcal B}_{n+1}^{(\alpha)} =
\sum_{j=1}^{n+1} \binom{n+1}{j} (-1)^{j+1} {\mathcal B}_{n+1-j}^{(\alpha)} + \frac{(-\beta)^{n+1} {\mathbb E}[T^{n+1}]}{(\alpha+n+1)_{n+1}}.
\label{recursionB}
\end{equation}

Due to the orthogonality property of generalized Laguerre polynomials, the approximation   $\hat{g}_n(t)$  has nice properties, that are:
\begin{equation}
\int_0^{\infty}  \hat{g}_n(t) {\rm d}t = 1
\label{(norm)}
\end{equation}
for all $n \geq 0,$  and the first $n$ moments of
$\hat{g}_n(t)$ are the same of $g(t).$ 
Unfortunately, $\hat{g}_n(t)$ is not guaranteed to be a pdf since negative values can occur. Indeed the values assumed by the polynomial $p_n(t)$ are not necessarily non-negative. However $p_n(t)$ may hold non-negative values on the tails and in a right-handed neighborhood of the origin, depending on the sign of some coefficients in \eqref{expansion3}. These conditions are established  in the following proposition. 
\begin{prop}\label{prop:hnpos} Suppose $p_n(t) > 0$ for all $t > 0$. Then $(-1)^nh_{n,n} \ge 0$ and $h_{n,0} \ge 0$. Conversely, if  $h_{n,0} > 0$ and $(-1)^n h_{n,n} > 0$, there exist $t_1 > 0$ and $t_2 > 0$ such that $p_n(t) > 0$ in $(0,t_1) \cup (t_2, +\infty)$, with $t_1$ and $t_2$ not necessarily distinct or finite.
\end{prop}
\begin{proof}
From \eqref{expansion3} we have
\begin{equation*}
p_n(0) = h_{n,0} \text{ and } p_n(t) \sim (-1)^n h_{n,n} t^n \frac{\beta^n}{n!},\text{ for }t \rightarrow \infty.
\end{equation*}
Since $\beta > 0$ and 
$p_n(t) \in C(0,\infty),$ the results follow from 
the sign permanence theorem.
\end{proof} 
Integrating $\hat{g}_n(t)$ in 
\eqref{expansion3} over $(0,t)$, an approximation of the FPT is
\begin{equation}
\hat{G}_n(t) = \frac{1}{\Gamma(\alpha+1)} \sum_{k=0}^n \frac{(-1)^{k}}{k!} h_{n,k} \left[ \Gamma(\alpha+k+1) - \Gamma (\alpha+k+1,\beta t) \right] \label{(cdf)}
\end{equation}
where $\Gamma(a,t) = \int_t^{\infty} \tau^{a-1} e^{-\tau} {\rm d} \tau$ is the incomplete Gamma function.  

\subsection{On the order $n$ of the approximation}
\label{napprox}
The normalization condition \eqref{(norm)} has been used to derive a first stopping criterion \cite{di2023approximating}. Indeed \eqref{(norm)} is equivalent to (Proposition 4.2 \cite{di2023approximating})  
\begin{equation}
h_{n,0}+\sum_{i=1}^n \frac{(-1)^i}{i!} h_{n,i} (\alpha +i)_i = 1.
\label{eq:cond1}
\end{equation}
As a result, in \eqref{expansion3} the order $n$  is increased as long as \eqref{eq:cond1} is satisfied with a fixed level of tolerance. Here, with the more conservative aim of obtaining reliable approximations, we propose to join the normalization condition \eqref{eq:cond1} with an additional condition following from Proposition \ref{prop:hnpos}. This condition guarantees an order $n$ of approximation such that $p_n(t)$ is positive close to the origin and as $t \rightarrow +\infty.$
Therefore, the recursive procedure runs as long as
\begin{equation}\label{eq:stop}
(|\hat{h}_{n+1}-1|>\varepsilon,\;\mbox{ for a fixed } \  \varepsilon>0) \mbox{ or }  (h_{n,0} > 0 \text{ and } (-1)^n h_{n,n} > 0)
\end{equation}
is fulfilled. 
\subsection{On the choice of $\alpha$ and $\beta$} \label{choiceAB}
Denote with $X_{\alpha,\beta}$ the rv having pdf $f_{\alpha,\beta}$ in \eqref{perprop31}. In \cite{di2021cumulant}  we have chosen
\begin{equation}
\label{eq:condsuaeb}
\alpha = \frac{c_1^2[T]}{c_2[T]}-1 \quad \text{and }\quad \beta = \frac{c_1[T]}{c_2[T]}
\end{equation}
since with these choices we have ${\mathcal B}_1^{(\alpha)}=
{\mathcal B}_2^{(\alpha)}=0$ in \eqref{(coeffak1)} and 
\begin{equation*}
 \mathbb E[X_{\alpha,\beta}] = \frac{\alpha+1}{\beta} = \mathbb E[T] 
 \quad \text{and }\quad 
 \mathbb E[X_{\alpha,\beta}^2] = \frac{(\alpha+1)(\alpha+2)}{\beta^2} = \mathbb E[T^2].
\end{equation*}  
Let us underline that a range of values was investigated for $\alpha$ and $\beta$. Actually the choices in \eqref{eq:condsuaeb} seem to be the most reliable with respect to the selection of parameters in \eqref{eqn: defn feller}. On the other hand, this choice falls  within the classical method of moments and is also suggested in \cite{provost2016distribution}, where a general procedure is developed for the approximation of a pdf based on its moments.  Different choices are suggested in \cite{belt1997optimal} where results concerning the determination of the two parameters $\alpha$ and $\beta$ are presented. Unfortunately, adopting these choices requires a knowledge of the FPT pdf not depending on $\alpha$ and $\beta$, which is not true in the case of CIR process. Moreover, the choices of $\alpha$ and $\beta$ in \eqref{eq:condsuaeb} return
\begin{equation*}
c_{\nu}[X_{\alpha,\beta}] = c_{\nu}[T] = \frac{\sqrt{c_2[T]}}{c_1[T]}
\end{equation*}
where $c_{\nu}$ denotes the coefficient of variation. In such a case the first equation in \eqref{eq:condsuaeb} reads $\alpha+1=(c_{\nu}[T])^{-2}.$ Thus an higher coefficient of variation reduces $\alpha + 1$ and increases the chance of a vertical asymptote of the gamma pdf in $0.$ As $g(0)=0,$ the occurrence of this vertical asymptote represents a numerical issue which is further worsened if the FPT pdf is flat with a large mean value and a heavy right tail with a large variance. To deal with this scenario, a successful strategy is the employment of a suitable standardization technique. The idea is to construct the  approximation $\tilde{g}_{n}(t)$ of the pdf $\tilde{g}(t)$ corresponding to
\begin{equation}
\tilde{T}=T/\sigma_T,
\label{stand}
\end{equation} 
where $\sigma_T$ is the standard deviation of $T.$ The approximated FPT pdf and cdf can be recovered as 
\begin{equation*}
\hat{g}_{n}(t) = \frac{1}{\sigma_T} \tilde{g}_{n}\left(\frac{t}{\sigma_T}\right) \quad {\rm and} \quad \hat{G}_{n}(t) = \tilde{G}_{n}\left(\frac{t}{\sigma_T}\right) 
\end{equation*} 
respectively. 
As $c_2[\tilde{T}]={\rm Var}[\tilde{T}] =1,$ from \eqref{eq:condsuaeb} the parameters of the gamma pdf are
\begin{equation*}
\tilde{\alpha} :=  ({\mathbb E}[\tilde{T}])^2-1 \,\, \text{and }\,\, \tilde{\beta}:= {\mathbb E}[\tilde{T}] 
\end{equation*}
so that $\tilde{\alpha} + 1 = \tilde{\beta}^2 = (c_{\nu}[X_{\tilde{\alpha},\tilde{\beta}}])^{-2} =  (c_{\nu}[\tilde{T}])^{-2}$
and ${\rm Var}[X_{\tilde{\alpha},\tilde{\beta}}]=1.$ 
The advantage of this strategy is to use an initial guess (the gamma pdf) with a shape resembling the picked shape of the FPT pdf and concentrating probability mass on small time values. Moreover, the moments ${\mathbb E}[{\tilde T}^n]$ grow slower than the moments ${\mathbb E}[T^n],$ leading to an observed improvement of numerical stability. 

Since the pdf $g(t)$ has support $(0, \infty),$ further information on the shape of the density can be recovered using some dispersion indices that work as the coefficient of variation but provide further global statistical information
\cite{kostal2011variability}.  In particular we consider the relative entropy based dispersion coefficient
\begin{equation}
c_h := \frac{\sigma_h}{\mathbb{E}(T)} = \frac{1}{\mathbb{E}(T)} \exp
\left\{\int_0^{\infty} g(t) \ln g(t) \, {\rm d}t -1\right\} .
\label{ch}
\end{equation}
The value of $\sigma_h$ quantifies how evenly is the pdf over $(0,\infty).$ Moreover, in logarithmic scale, $c_h$ is inversely proportional to the  Kullback-Leibler distance of the pdf $g(t)$ from the exponential density with mean ${\mathbb E}[T].$ For densities resembling the exponential distribution, the coefficients $c_{\nu}$ and $c_h$ are approximately equal to $1$.  
\section{Computational issues} \label{section4}
In \cite{di2023approximating} a fast recursive procedure for implementing \eqref{expansion3} is proposed, relying on nested products and taking advantage of \eqref{rech} and \eqref{recursionB}. 
Note that, differently from the geometric Brownian motion in \cite{di2023approximating}, in the case of the CIR process an additional issue arises in computing \eqref{recursionB}. The FPT moments are calculated from the FPT cumulants \eqref{cumulants}
using the recursion \eqref{eq:recursionmoments}. Because of the series involved in \eqref{cumulants},  an approximation order $m$ should be chosen before computing the moments through \eqref{eq:recursionmoments}.
Here, a standard approach has been used, which involves computing partial sums of the series as long as their difference exceeds an input tolerance.

\subsection{On the monotonicity of $\hat{G}_n$}\label{sec:pos1}
For a fixed $\Delta t >0,$ the computation of the FPT cdf can benefit from the iterative calculation of increments
\begin{equation}
\label{increment}
\Delta \hat{G}_n(t) = 
\frac{1}{\Gamma(\alpha+1)} \sum_{k=0}^n \frac{(-1)^{k}}{k!} h_{n,k} \left[ \Gamma(\alpha+k+1, \beta t) - \Gamma (\alpha+k+1,\beta (t+\Delta t)) \right]
\end{equation}
where $\Delta \hat{G}_n(t) = 
\hat{G}_n(t + \Delta t) - \hat{G}_n(t), t > 0.$
Note that the increments $\Delta \hat{G}_n(t)$ might be negative, in accordance with the values of $t$ where $\hat{g}_n(t) <0.$
As a by-product, the approximated cdf may turn out to be decreasing in these intervals. 
Moreover, if in a right-handed neighborhood of the origin the first increments are already negative, this circumstance might determine negative values of the approximated cdf itself in the same neighborhood.  
A possible correction to this last drawback is: set  $\tau_0 = 0$  or 
$\tau_0 =  \min\{ t > 0 | \Delta \hat{G}_n(t) <0 \},$ depending if $\Delta \hat{G}_n(\Delta t)$ is negative or not, and $\tau_1 = \min\{ t > \tau_0 | \hat{G}_n(t) > \hat{G}_n(\tau_0) \}.$   Then  iteratively find the intervals $[\tau_i, \tau_{i+1}]$ such that
$$\tau_i = \min \{ t > \tau_{i-1} | \Delta \hat{G}_n(t) < 0\} \quad 
{\rm and} \quad \tau_{i+1} = \min \{ t > \tau_{i} | \hat{G}_n(t) > \hat{G}_n(\tau_i) \}$$
in order to replace 
$\hat{G}_n(t)$ with a suitable line for $t \in [\tau_i, \tau_{i+1}].$ For the cases here considered  the intervals $t \in [\tau_i, \tau_{i+1}]$ have a small amplitude.  The advantage of this approach is twofold. The first one is getting an approximated cdf which is positive and increasing. The second advantage is the chance to use an ad hoc numerical procedure to  carry out the derivative of the resulting cdf and thereby automatically recovering an approximation of the pdf which is always positive. 
Since the corrected $\hat{G}_n(t)$ is linear for $t \in  [\tau_i, \tau_{i+1}],$ the corresponding pdf will be constant  in $[\tau_i, \tau_{i+1}].$
Thus, this approximation is computationally simple, but it might fail to fit some properties of the FPT pdf of a CIR process.
The following subsection suggests a different correction of the approximated pdf taking into account specifically these properties. 
 
\subsection{On the positivity of $\hat{g}_n$}\label{sec:pos2}
For a fixed order $n$ of approximation, 
it could be of interest constructing the pdf directly, overcoming the drawbacks occurred in the numerical derivation of the cdf. In that case, 
although the stopping criteria in \eqref{eq:stop} take into account Proposition \ref{prop:hnpos},
there is no guarantee that $\hat{g}_n(t)$ is non-negative on $(0, \infty)$  depending on the values assumed by $p_n(t)$ for $t \in (t_1, t_2).$ If this happens, an ad hoc strategy can be implemented to solve this issue. 

Suppose 
$(t_{1,\text{neg}}, t_{2,\text{neg}}) \subseteq (t_1, t_2)$
such that 
$\hat{g}_n(t) < 0, \, t \in (t_{1,\text{neg}}, t_{2,\text{neg}}).$
In what follows we develop a simple numerical procedure to replace $\hat{g}_n(t)$ with a suitable positive function $p(t)$, for $t$ in a generic interval $ (t'_1, t'_2)  \supseteq (t_{1,\text{neg}}, t_{2,\text{neg}}).$ 
It is reasonable to assume that $\hat{g}_n(t)$ can be negative in an interval located to the right or/and to the left of the approximated mode of the FPT rv, since the FPT pdf of a diffusion process is unimodal \cite{Uwe80}.
 In both cases, a classical technique would consist in selecting a fourth-degree polynomial $p(t)$   interpolating
 smoothly  $(t'_1,\; \hat{g}_n(t'_1))$ and $(t'_2,\; \hat{g}_n(t'_2)),$ fulfilling the additional constraints imposed by the conservation of probability mass
 \begin{equation}
\int_{t'_1}^{t'_2}p(t)\, {\rm d}t = \int_{t'_1}^{t'_2}\hat{g}_n(t)\, {\rm d}t,
\label{consvmass}
\end{equation}
as well as positivity and monotonicity. Since such a polynomial is unique, the last two remaining  conditions would possibly  be satisfied by a computationally cumbersome choice of $t'_1$ and $t'_2$. 
Therefore, in the following we propose a different approach
both to determine numerically $(t'_1,t'_2)$ and to correct $ \hat{g}_n(t),$ 
taking into account the conditions required on
the FPT pdf $g(t)$ in Theorem \ref{thm:a_k}.

Suppose $m^* = \max_{t \in (0,\infty)}{\hat{g}_n(t)}$ be the approximated mode of the FPT rv $T$. 
In agreement with the previous observations, the following two possible scenarios might occur:
\begin{enumerate}
\item[{\it a)}] $\hat{g}_n(t) < 0$ for $t \in (t_{1,\text{neg}}, t_{2,\text{neg}})$ with $t_{2,\text{neg}} < m^*$ 
\item[{\it b)}] $\hat{g}_n(t) < 0$ for $t \in (t_{1,\text{neg}}, t_{2,\text{neg}})$ with $t_{1,\text{neg}} > m^*.$
\end{enumerate}
Two procedures were therefore developed  taking into account the behavior of the FPT pdf either in a right-handed neighborhood of the origin - Case {\it a)} - and on the tail - Case {\it b)} -  with the aim of minimizing the number of parameters involved while streamlining the writing and the procedure. No conservation of the probability mass \eqref{consvmass} is required in this strategy. Empirically, this is justified by the circumstance that in all the observed cases the negative areas are so small that the mass involved gives almost no contribution.
\subsubsection*{Case a)}
Since $g(0)=0,$ we set $t'_1 \coloneqq 0.$  To reduce the number of parameters, we assume $p(t)=a t^{\delta}$ with $a>0$ 
and $\delta > \frac{\alpha}{2} + 1$ according to Theorem \ref{thm:a_k}. Therefore, the correction of $\hat{g}_n(t)$ is defined as
\begin{equation}\label{eq:corr1}
    \hat{g}^{corr}_{n}(t) = \begin{cases}
        a\,t^\delta & 0 \leq t \leq t'_2,\\
        \hat{g}_n(t) & t > t'_2.
    \end{cases}
\end{equation}
In order to  achieve a certain level of smoothness, $a$ and $\delta$ in \eqref{eq:corr1} are
chosen such that 
$$ p(t'_1) = \hat{g}_n(t'_1)\quad {\rm and} \quad 
\left. \frac{{\rm d}}{{\rm d} t}p(t)\right|_{t=t'_1} = \left. \frac{{\rm d}}{{\rm d} t}\hat{g}_n(t) \right|_{t=t'_1}.$$ 
Finally, we set
$$t'_2 \coloneqq \min \left\{ t \in (0,m^*)  \suchthat  \int_{0}^{t} \hat{g}_n(t) \, {\rm d}t > 0 \right\}$$
to avoid an excessive  increment of the probability mass 
when $\hat{g}_n(t)$ is replaced by $\hat{g}^{corr}_{n}(t).$ Fig.~\ref{fig:pos}-a) shows an example of 
negative $\hat{g}_n(t)$ in an area close to the origin together with its correction $\hat{g}^{corr}_{n}(t),$ as obtained by the procedure  described above.
\subsubsection*{Case b)}
We assume $p(t)=a e^{bt},$ with $a >0$ and $b<0,$ according 
to the well known exponential asymptotic behaviour of the FPT pdf of diffusion processes with steady-state distribution  \cite{nobile_ricciardi_sacerdote_1986}. Therefore, the correction of $\hat{g}_n(t)$ is defined as \begin{equation}
    \hat{g}^{corr}_{n}(t) = \begin{cases}\label{eq:corr2}
        \hat{g}_n(t) & t < t'_1\\
        a e^{bt} & t'_1 \leq t \leq t'_2 \\
        \hat{g}_n(t) & t > t'_2.
    \end{cases}
\end{equation}
In this case, $a$ and $b$ are chosen such that 
$$p(t'_1) = \hat{g}_n(t'_1) \quad {\rm and} \quad 
p(t'_2) = \hat{g}_n(t'_2).
$$
In order to fit a decreasing exponential function  in the interval $(t'_1, t'_2)$, the endpoints  $t'_1$ and $t'_2$ are chosen such that
$$ t'_2 = \min \left\{{t > t_2}  \suchthat \frac{{\rm d}}{{\rm d}t} \hat{g}_n(t) <0\right\} \quad {\rm and} \quad t'_1 = \min \left\{{t < t_1}  \suchthat \hat{g}_n(t) > \hat{g}_n(t'_2)\right\}.$$
Fig.~\ref{fig:pos}-b) shows an example of 
negative $\hat{g}_n(t)$ for values of $t$ larger than the mode together with its correction $\hat{g}^{corr}_{n}(t),$ as obtained by the procedure just described.
\begin{figure}[h]
\centering
\includegraphics[scale = 1]{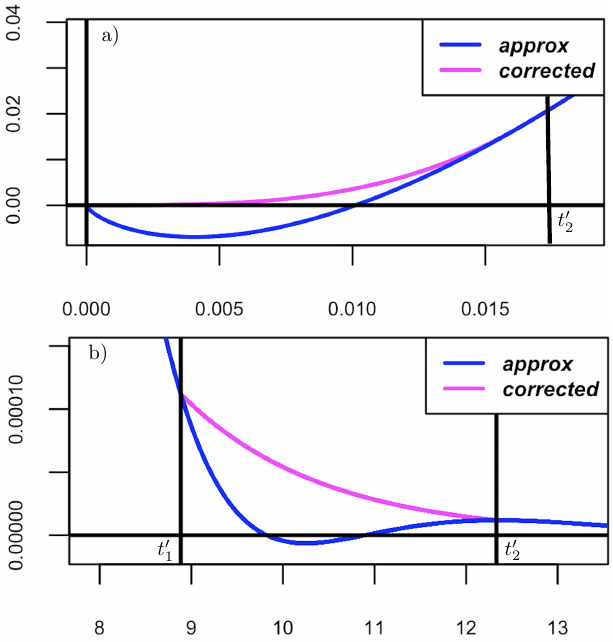}
\caption{In a) plots of the approximation $\hat{g}_n$ and of its correction $\hat{g}^{corr}_{n}$ \eqref{eq:corr1} 
over the interval $(0,t'_2)$ are given for $n = 10$ and parameters $y_0 = 0$, $\mu = 3$, $S=10$, $c = - 10$, $\sigma = 1.2, \tau = 0.2$ (see case C in Section \ref{numexamples}). In b), plots of the approximation $\hat{g}_n$ and of its correction $\hat{g}^{corr}_{n}$ \eqref{eq:corr2} over the interval $(t'_1,t'_2)$ are given for $n = 9$ and parameters $y_0 = 0.2$, $\mu = 0.9$, $S =1$, $c = 0$, $\sigma = 1.2, \tau = 2/3$ (see case A in Section \ref{numexamples}).
}\label{fig:pos}
\end{figure}
\section{\label{section5}
Numerical results } 
In this section, the functions $\hat{G}_n(t)$ in \eqref{(cdf)} and $\hat{g}_n(t)$ in \eqref{expansion3} 
are used to approximate respectively the FPT cdf $G(t)$ and the FPT pdf $g(t)$ of a CIR process \eqref{eqn: defn feller} with the corrections suggested in subsections \ref{sec:pos1} and \ref{sec:pos2}. 

Before examining the numerical results, in the following we provide some details on how the goodness of approximation was assessed. 
\subsection{Comparisons with alternative approximation methods  
} \label{comparison}

As the shape of the FPT pdf and cdf for a CIR process is unknown, the proposed approximations' validity needs to be evaluated by comparing it  with alternative estimates obtained using different techniques.  

For one-dimensional diffusion processes, the FPT pdf through a time-dependent boundary can be recovered as the solution of a Volterra integral equation of the second kind \cite{buo_integral}. With suitable numerical methods for approximating the integral, a discrete numerical evaluation of this solution can be efficiently computed.
Unfortunately, when implementing these tools, we came upon the issue that, when the coefficient of variation is large, 
this procedure 
is subject to overflow failures and could lead to completely misleading results. This circumstance is likely  amplified by an unavoidable  propagation of numerical errors  since the new approximated values
of the FPT pdf are computed using those recovered at the previous steps.
However, even for FPT pdfs with a small coefficient of variation, we encountered numerous issues in its implementation. These undesirable behaviors are essentially caused by the presence of the Bessel function
in the transition pdf of the process, whose derivatives cause numerical overflow issues. Similar difficulties have been encountered using the {\tt R} package {\tt fptdApprox} \cite{roman2fptdapprox}. Ultimately, beyond these pathological cases, the numerical results are 
comparable to those obtained by Monte Carlo methods, which we have therefore chosen as methods for assessing the goodness of approximation. 

A Monte Carlo method consists in generating
sample paths of the CIR process and look for their FPTs over the given threshold. When choosing how to sample the paths of the CIR process in the Monte Carlo method, we first implemented the Milstein algorithm
\cite{SDE} which generates a trajectory by a
suitable discretization of the stochastic differential equation \eqref{eqn: defn feller}. As it is well-known, this procedure is time-demanding:  to get a FPT sample of size $N,$ at least $N$ different trajectories of
the CIR process must be generated. Indeed not all the generated trajectories may reach the threshold in a reasonable time. This also implies that the FPT pdf can be underestimated if a
finite time interval has been set for the simulation, as usually happens. Moreover, the fixed time step  determines how accurately the dynamics can be described and the computational time increases as this time step gets smaller.
Therefore, to obtain significant results, it is necessary to choose a very small step size and simulate many trajectories of the CIR process.
This can be very time-consuming, especially if the coefficient of variation of $T$ is large, as a consequence of the likelihood of a large time span length over which the trajectories must be simulated.

Similar problems arise
when sampling positions of the CIR process using its transition pdf, a non-central chi-square distribution \cite{fel51}. In such a case, starting from $Y_0 = y_0,$ an instance of $Y_ {\Delta t}$ is  generated from the conditional distribution
of $Y_ {\Delta t}|Y_0 = y_0,$ an instance  of $Y_ {2 \Delta t}$ from the conditional distribution of $Y_ {2 \Delta t}|Y_ {\Delta t}$ and so on. The results obtained by the two Monte Carlo methods are comparable. 
However we have used the Milstein algorithm, since the computational time is lower in all the cases examined.  

Once a FPT random sample has been collected, 
a sufficiently reliable estimate of the cdf shape can be obtained using the empirical cdf. 
Nonparametric methods can also be used to obtain an estimate of the pdf.  The most widely used nonparametric  pdf estimator is the histogram. Despite its popularity, the drawbacks of this tool are well known in the literature, as for example, the strong dependence on bandwidth choice. In the literature, KDEs are typically mentioned as simple alternatives to histograms. If the unknown pdf has its support confined to the positive half line and is not smooth at the origin, then the kernel method can perform not efficiently \cite{hall1980estimating}. Indeed, the mode of the unknown pdf may be actually  hidden by the KDE assigning positive mass to negative values (see Fig.~\ref{fig:histdenB}).

To recover the smoothness characterising the KDEs and still obtain an adequate estimated density with support $(0, \infty)$, an estimator based on an orthogonal series can be very competitive \cite{efromovich2010orthogonal}. Indeed, 
an orthogonal series estimator is exactly what is obtained from the approximation $\hat{g}_n(t)$ in \eqref{approximationgen2bis} when the theoretical FPT moments are replaced by their corresponding sample moments. In fact, suppose a sample of iid FTPs $\mathcal{T} = \{T_1, \ldots, T_N\}$ is available, arising either from simulations or from experiments.  
A straightforward calculation shows that replacing FPT moments
(\ref{(momT)}) with sample moments  is equivalent to replacing ${\mathbb E}[L_k^{(\alpha)}(\beta T)]$ in (\ref{(coeffak1)}) with its sample mean estimator
$$\bar{l}_k = \frac{1}{N} \sum_{i=1}^N L_k^{(\alpha)}(\beta T_i).$$
Then, the FPT pdf $g(t)$ can be approximated by 
\begin{equation}
\tilde{g}_n(t) =f_{\alpha,\beta}(t) \bigg( 1 + \sum_{k = 1}^n   \bar{l}_k b_k^{(\alpha)}   L_k^{(\alpha)}(\beta t) \bigg) \quad {\rm with} \quad  b_k^{(\alpha)} = \frac{k! \, \Gamma(\alpha+1)}{ \Gamma(\alpha+1+k)} 
\label{seriesestimatornostro}
\end{equation}
which is the orthogonal series estimator of $g(t).$
This observation reveals an additional advantage of using the Laguerre-Gamma approximation. If the FPT moments/cumulants are not known but a random sample is available, the Laguerre-Gamma approach offers the opportunity to recover an approximation of the FPT pdf similarly to the orthogonal series estimators. In such a case, the estimates carried out
by sample moments or by $k$-statistics \cite{kstatistics} replace the occurrences of FPT moments or cumulants  respectively.
Under suitable hypotheses on the true pdf $g(t),$ estimations of  the convergence order of $\tilde{g}_n(t)$ to $g(t)$ are assessable through the mean integrated squared error \cite{hall1980estimating}.  Indeed, by recalling that the generalized Laguerre polynomial sequence is orthonormal with respect to the reference density  $f_{\alpha,\beta}(t),$ the mean integrated squared error is  \cite{diggle1986selection}
\begin{eqnarray}
J(n) & = & \mathbb E \int_0^{\infty} \left|\frac{\tilde{g}_n(t) - g(t)}{f_{\alpha,\beta}(t)}\right|^2 f_{\alpha,\beta}(t) \,\, {\rm d} \, t \label{MISEnostro} \\
& = &  \frac{1}{N} \sum_{k=1}^n b_k^{(\alpha)} {\rm Var}[L_k^{(\alpha)}(\beta T)] + \sum_{k > n} b_k^{(\alpha)} \mathbb E[L_k^{(\alpha)}(\beta T)]^2
\nonumber
\end{eqnarray}
with $b_k^{(\alpha)}$ as in \eqref{seriesestimatornostro}.  Thus increasing the sample size leads to a reduction of the error $J(n)$  as expected from the estimation of moments with sample moments. There are various strategies in the literature to choose the degree of the polynomial approximation $n$ in \eqref{MISEnostro}, see for example \cite{diggle1986selection}. A discussion on which strategy is the most effective for the CIR  process goes beyond the scope of this paper. 

In all the cases examined, the results estimated by the orthogonal series method on a collected FPT sample overlap with the Laguerre-Gamma approximations \eqref {expansion3}, when the stopping criteria addressed in Section \ref{napprox} are used. This explains why the subsequent section does not include these results. Instead, to assess the goodness of these approximations, 
the corresponding histograms have been used 
despite their well-known flaws.

\subsection{Numerical examples} \label{numexamples}
To analyse the efficiency and the usefulness of the proposed method, in the following  we consider three  scenarios:
\begin{enumerate}
\item[{\rm case} $A$:] $y_0 = 0.2$, $\mu = 0.9$, $S =1$, $c = 0$, $\sigma = 1.2$ and $\tau = 2/3$,
\vskip0.3cm
\item[{\rm case} $B$:] $y_0 = 0.01$, $\mu = 0.005$, $S = 0.02$, $c = 0$, $\sigma = 0.1$ and $\tau = 0.25$,
\vskip0.3cm
\item[{\rm case} $C$:] $y_0 = 0$, $\mu = 3$, $S=10$, $c = - 10$, $\sigma = 1.2$ and $\tau = 0.2.$
\end{enumerate}
Each case results in FPT pdfs and cdfs with different forms and statistical properties, as shown in Fig.~ \ref{fig:ABC_density}, where plots of empirical FPT cdfs are given.
\begin{figure}[h]
\centering
\includegraphics[scale = 0.42]{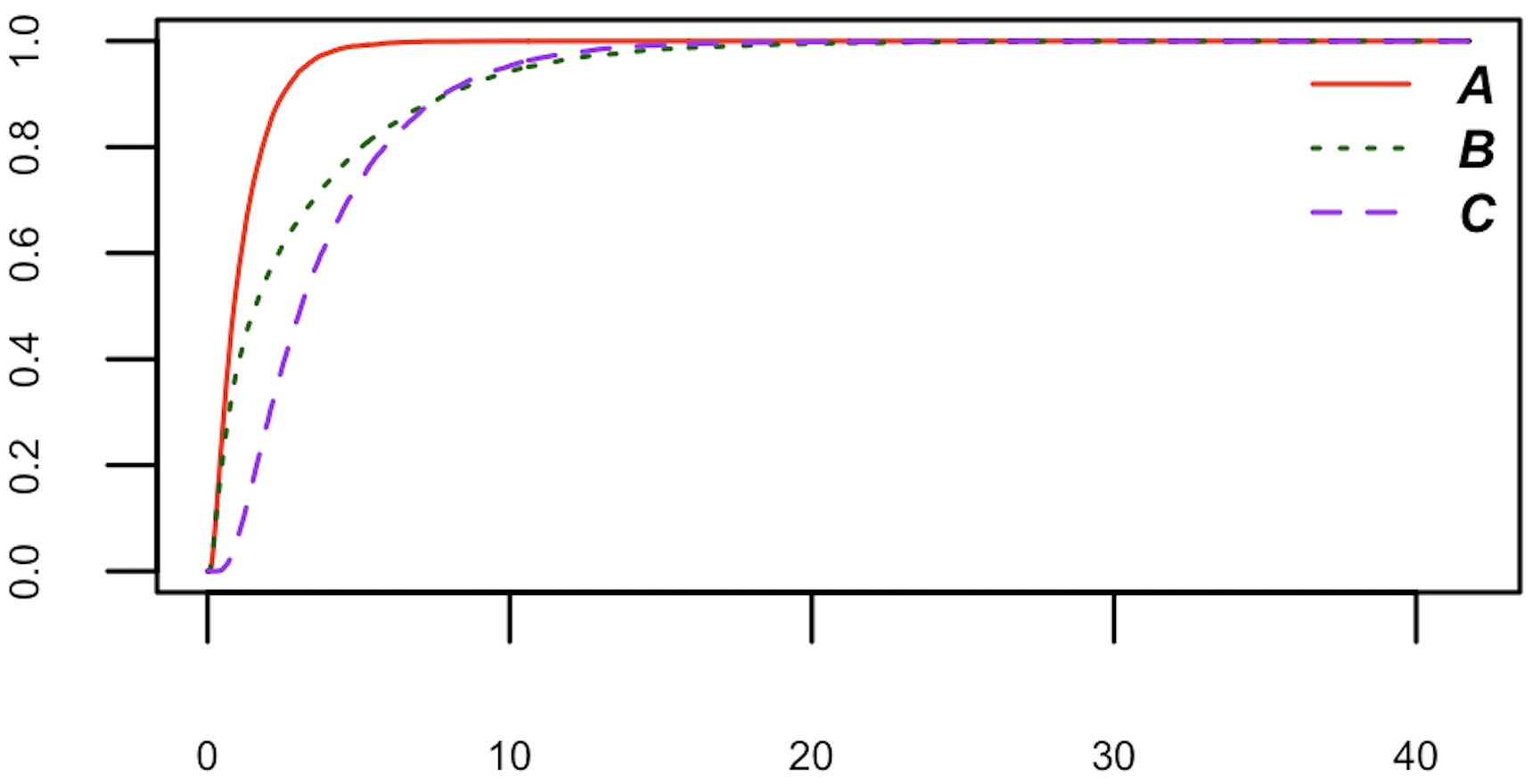}
\caption{Plots of the empirical (not standardized) FPT cdfs for cases A (in solid red), B (in dash green), and C (in dashed purple).}
\label{fig:ABC_density}
\end{figure}

According to Section \ref{comparison}, the empirical cdfs have been constructed after using
the Milstein method 
to simulate a sample of $10^4$ FTPs for each case. In the following, these three samples are denoted by $\mathcal{T}_A$, $\mathcal{T}_B$ and $\mathcal{T}_C$ respectively. 

For each case, we have computed the FPT dispersion coefficients as given in Section \ref{choiceAB}.  
The coefficient of variation is computed using the theoretical FPT mean and variance. The estimated relative entropy based dispersion coefficient $\hat{c}_h$ is computed with the Vasicek estimator \cite{kostal2012nonparametric} using the samples $\mathcal{T}_A$, $\mathcal{T}_B$ and $\mathcal{T}_C$ respectively. The results are given in Table \ref{Table1}.

\begin{table}[t!]
\centering
\caption{\footnotesize{FPT dispersion indexes $c_v$ and $\hat{c}_h$ togheter with mean, standard deviation, skewness and kurtosis for the cases A, B and C.} } 
\label{Table1}
\begingroup\small
\begin{tabular}{lrrrrrrr}
  \hline
  \hline
  & $c_v$ & $\hat{c}_h$  & ${\mathbb E}[T]$ & $\sqrt{{\rm Var}[T]}$ &  $\gamma_1$ & $\kappa_1$  \\ \hline 
 A & 0.855 & 0.909 & 1.16 & 0.984 & 1.968 & 5.9862\\ 
  B & 1.231 &0.916 & 2.991 & 13.56 &  2.39 & 8.118\\ 
  C & 0.765 & 0.855 & 3.937 & 9.084 & 1.905 & 5.572\\  
   \hline
\end{tabular}
\endgroup
\end{table}

\begin{figure}[h]
\centering
\includegraphics[scale = 0.7]{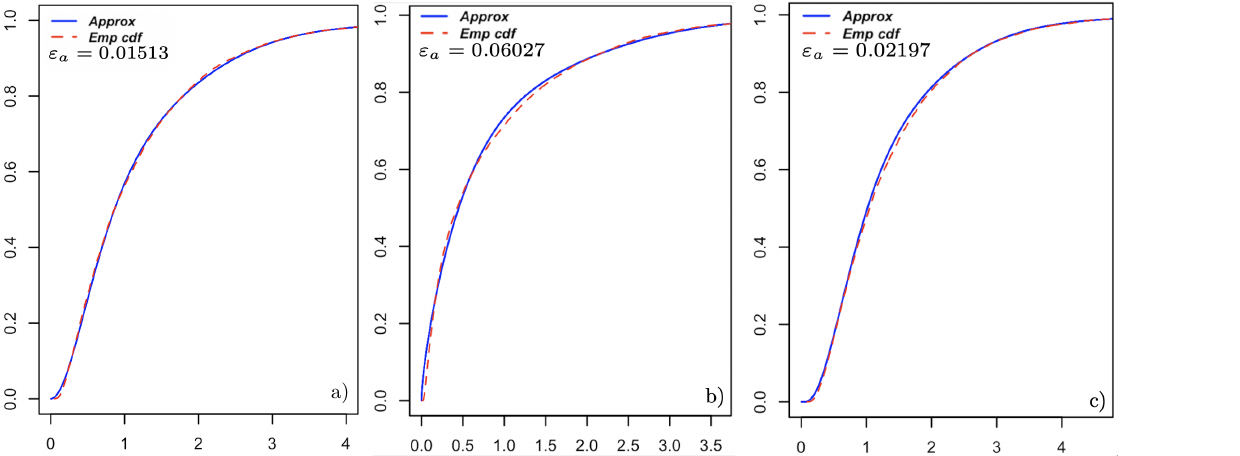}
\caption{Plots of the approximated $\tilde{G}_n(t)$ (in solid blue) and of the empirical cdf (in dashed red) together with the corresponding absolute error $\varepsilon_{a}.$ 
The plots refer: to case A in a) with $n = 10$, $\alpha = 0.367$ and $\beta = 1.17$; to case B in b) with $n = 10$, $\alpha = -0.34$ and $\beta = 0.812$;  to case C in c) with $n = 9$, $\alpha = 0.7$ and $\beta = 1.306$. Note that $\tilde{G}_n(t)$ is obtained using the stopping criteria \eqref{eq:stop} and corrected according to Section \ref{sec:pos1} while the empirical cdf is obtained from the  standardized samples ${\mathcal T}_A, {\mathcal T}_B$ and ${\mathcal T}_C$ respectively.}\label{fig:ABCcdf}
\end{figure}

Figs.~\ref{fig:ABCcdf},~\ref{fig:histdenA},~\ref{fig:histdenB} and \ref{fig:histdenC} refer to the standardized FPT rv $\tilde{T}$ in \eqref{stand}. 
In Fig.~\ref{fig:ABCcdf} we have plotted the empirical cdfs $G_e(t)$, corrisponding to the samples $\mathcal{T}_A$, $\mathcal{T}_B$ and $\mathcal{T}_C,$ together with the approximated cdfs $\hat{G}_n(t),$ as obtained using \eqref{increment} and the corrections described in Section \ref{sec:pos1}, for $\Delta t = 10^{-5},$
normalized by the corresponding standard deviations (see Table \ref{Table1}).  Moreover each figure displays the absolute error defined as $\varepsilon_{a} = \max_{t \geq 0}
|\hat{G}_n(t) - G_e(t)|.$
 Figs.~\ref{fig:histdenA}, \ref{fig:histdenB} and \ref{fig:histdenC}, correspond to cases A, B and C respectively. To emphasize the differences in density estimations, as discussed in Section \ref{comparison},  we have plotted in these figures a classical KDE\footnote{The  KDE has been generated by the {\tt R} function {\tt density()} \cite{R}.}, a histogram\footnote{The histogram has been generated by the {\tt R} function {\tt hist()} \cite{R}.} and the standardized approximated pdf $\tilde{g}_n(t),$ corrected according to  Section \ref{sec:pos2}.  Since the three considered instances  correspond to FPT pdfs with different shapes, these comparisons should provide a comprehensive picture of the strengths and weaknesses of the proposed method, which are discussed in the following.

\begin{figure}[h]
\centering
\includegraphics[scale = 0.35]{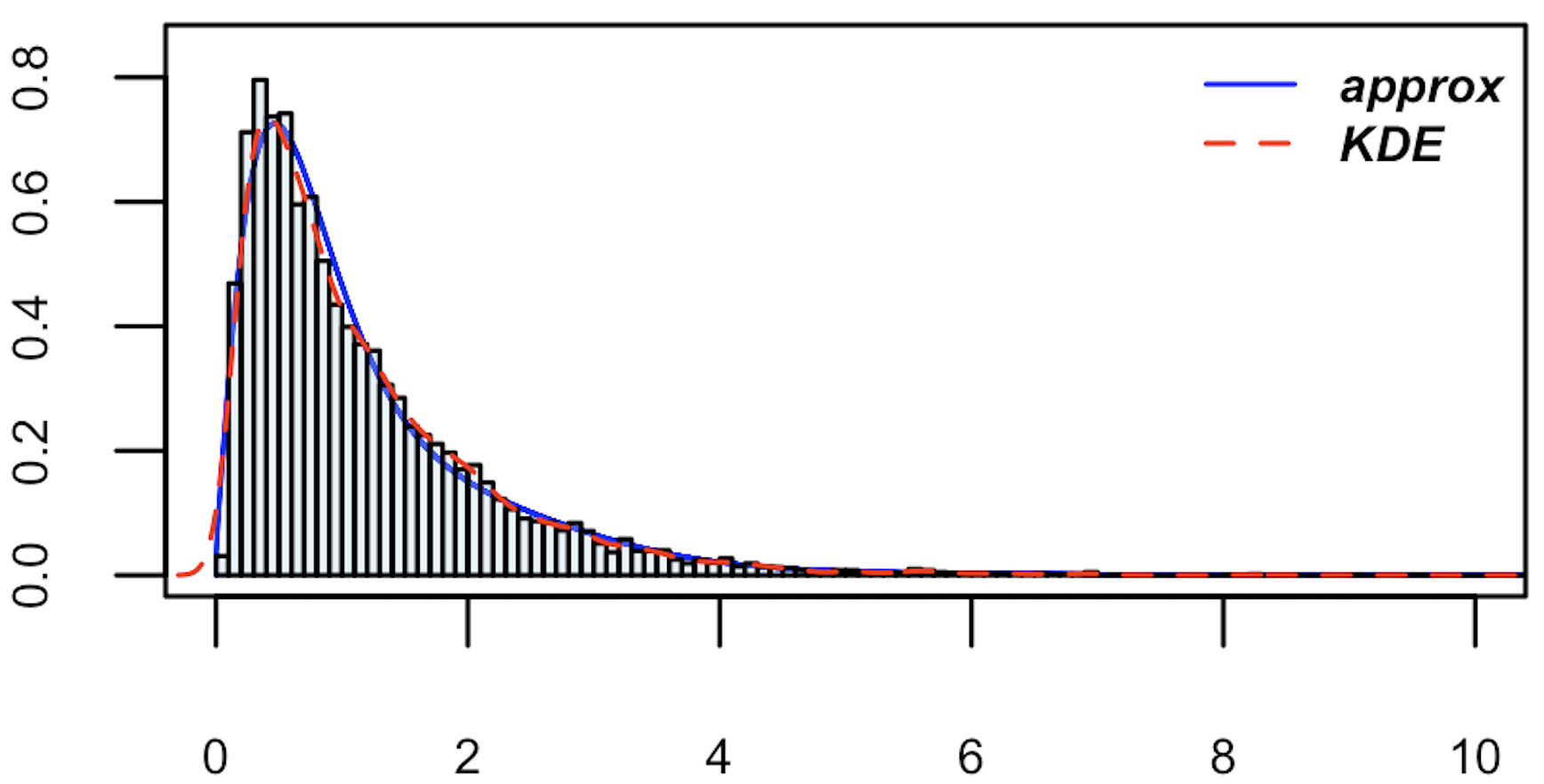}
\caption{Plot of $\tilde{g}_n(t)$ (in solid blue) in case A with $n = 10$, $\alpha = 0.367$ and $\beta = 1.17$, obtained with the stopping criteria \eqref{eq:stop} and corrected to ensure positivity as in \eqref{eq:corr2}, together with a KDE (in dashed red) and a histogram both computed with the standardized sample $\mathcal{T}_A$.}\label{fig:histdenA}
\end{figure}

\begin{figure}[h]
\centering
\includegraphics[scale = 0.5]{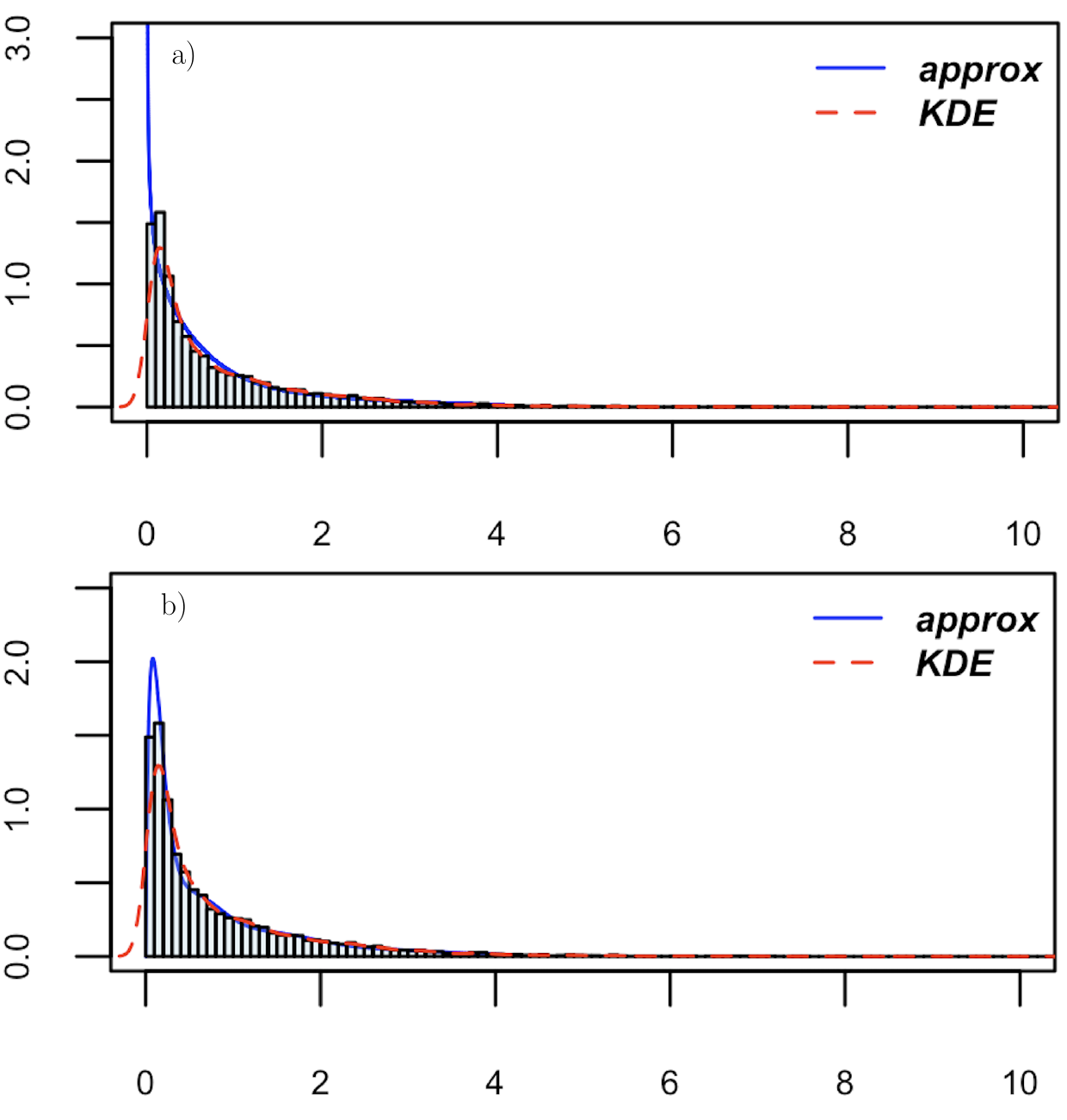}
\caption{In a), plot of $\tilde{g}_n(t)$ (in solid blue) in case B with $n = 10$, $\alpha = -0.34$ and $\beta = 0.812$, obtained with the stopping criteria \eqref{eq:stop} together with a KDE (in dashed red) and a histogram both computed with the standardized sample $\mathcal{T}_B$. In b), a plot of $\tilde{g}_n(t)$ (in solid blue) in case B with $n = 55$ and $\alpha = -0.34$, obtained without any stopping criterion, increasing the numerical precision. 
}\label{fig:histdenB}
\end{figure}

\begin{figure}[h]
\centering
\includegraphics[scale = 0.35]{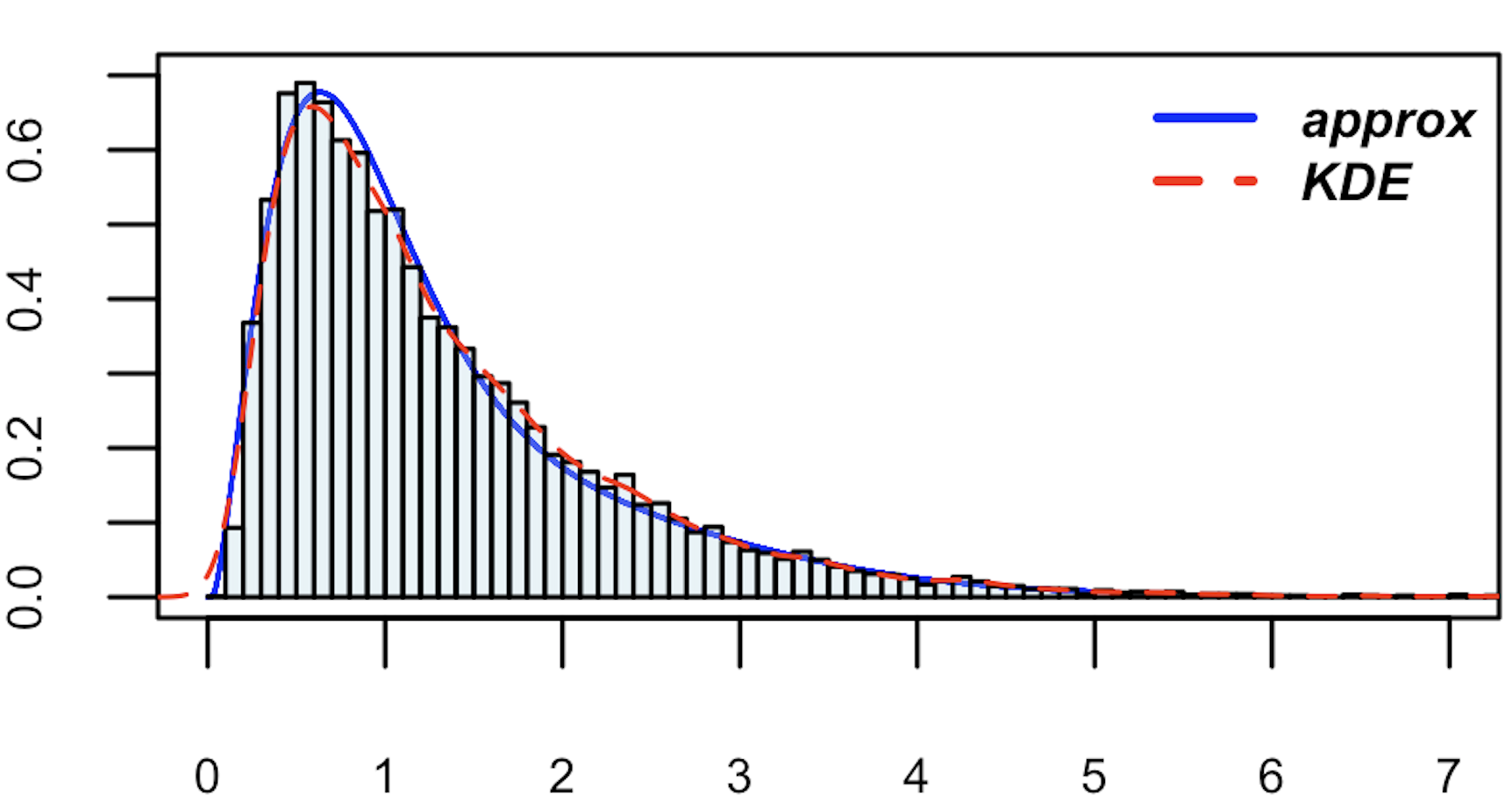}
\caption{Plot of $\tilde{g}_n(t)$ (in solid blue) for case C with $n = 9$, $\alpha = 0.7$ and $\beta = 1.306$, obtained with the stopping criteria \eqref{eq:stop} and corrected to ensure positivity as in \eqref{eq:corr1}, together with a KDE (in dashed red) and a histogram both computed with the sample $\mathcal{T}_C$, after the latter has been standardized.}\label{fig:histdenC}
\end{figure}

\subsubsection{Case A}  Among the three instances taken into consideration, case A has the lightest tails (see  Fig.~ \ref{fig:ABC_density}). This produces an accurate approximation, even with a small value of $n.$ Indeed there is a low absolute error
between the empirical and approximated cdf (see Fig.~ \ref{fig:ABCcdf}-a)), which is assumed at $t =  0.134.$ For the pdf, the suggested approach combined with the stopping criteria \eqref{eq:stop} yield an approximation $\tilde{g}_n(t)$ with $n=10$, $\alpha = 0.367$ and $\beta = 1.17$. In this case, $\tilde{g}_n(t)$ is negative on a small interval after the mode, as shown in Fig.~\ref{fig:pos}. Therefore, a suitable correction of $\tilde{g}_n(t)$ has been implemented, as outlined in Section \ref{sec:pos2}. This
corrected approximation is plotted in Fig.~\ref{fig:histdenA}. The same figure shows also the estimated pdf obtained with a classical KDE and with a histogram, both computed on the standardized sample $\mathcal{T}_A.$
\subsubsection{Case B} In this case, different considerations are required. From Table \ref{Table1}, the FPT rv has a coefficient of variation larger than 1. This makes the approximation more challenging because the distribution has a significantly heavy tail (see Fig.~ \ref{fig:ABC_density}). 
The stopping criteria \eqref{eq:stop}
yield an approximated FPT pdf $\tilde{g}_n(t)$ with $n = 10$, $\alpha = -0.34$ and $\beta = 0.812$, not adequately reproducing $\tilde{g}(t)$ when compared with a KDE and a histogram, as shown in Fig.~\ref{fig:histdenB} a).  This is a result of the first conservative stopping criterion in \eqref{eq:stop}, which was initially proposed to avoid numerical instability caused by an excessively high order of approximation.  To underline that the behavior of $\tilde{g}_n$ is significantly influenced by the numerical precision, Fig.~ \ref{fig:histdenB}-b) shows the remarkably good approximation obtained when computing $\tilde{g}_n(t)$ with a very high numerical precision, allowing to push the recursion up to $n=55$. The numerical precision has been raised using the {\tt R}-package \texttt{bignum} \cite{big23}. Still, it is worth to mention that the approximation $\tilde{g}_n(t)$ in Fig.~ \ref{fig:histdenB}-a) yields a satisfactory result for the tail of $\tilde{g}(t)$. Note that the absolute error at $t= 0.042$ between empirical and approximated cdf is higher compared to case A  (see Fig.~ \ref{fig:ABCcdf}-b)). This results from overestimating $\tilde{g}(t)$ in the  right-handed neighborhood of the origin, see Fig.~\ref{fig:histdenB}-b). 

\subsubsection{Case C} 
In this case, from Fig.~\ref{fig:ABC_density} and Table \ref{Table1},  the FPT pdf has a coefficient of variation less than 1 along with a tail whose heaviness lies between cases A and B.  As in case A, this value of the coefficient of variation should intuitively ensure a good approximation.  Indeed there is a low absolute error between the empirical and approximated cdfs (see Fig.~\ref{fig:ABCcdf}-c)), which is assumed at $t = 1.365.$ For the pdf, the suggested approach yields an approximated FPT pdf $\tilde{g}_n(t)$ with $n=9$, $\alpha = 0.7$ and $\beta = 1.306$. In this case, $\tilde{g}_n(t)$ is negative on a small interval before the mode and close to the origin (see Fig.~ \ref{fig:pos}). Therefore, also in this case, a suitable correction of $\tilde{g}_n(t)$ has been implemented, as outlined in Section \ref{sec:pos2}. The resulting approximation is shown in Fig.~ \ref{fig:histdenC}.

\section{Application: 
An acceptance-rejection type algorithm} \label{section 5}
In this section we propose a possible application of the polynomial FPT pdf approximation consisting in sampling FPTs using a method similar to the acceptance-rejection one. 

The acceptance-rejection method is a classical technique for sampling from a distribution that is unknown or difficult to simulate through an inverse transformation \cite{Robert1999monte}. Under such circumstances, samples are collected from an auxiliary density if a suitable probability of acceptance is known \cite{chib1995understanding, casella2004generalized}.

More in details, let  $Z$ be a rv  with pdf \(\pi\). Suppose there exists a constant \(M > 0\) and a pdf \(q(x)\) such that 
\begin{equation}
\label{ARcondition}
q(x) >0 \quad \text{whenever} \quad  \pi(x) > 0 \quad \text{and} \quad \frac{\pi(x)}{q(x)} \le M, \quad \forall x \in \operatorname{supp}(Z).
\end{equation}
The acceptance-rejection method exploits the  condition in \eqref{ARcondition} to sample from the support of $Z$.
Therefore, if the FPT pdf $g(t)$ were known and an upper bound for the right hand side of 
\begin{equation} \label{acceptance} 
\frac{g(t)}{f_{\alpha,\beta}(t)} = \sum_{k\ge 0} a_k^{\alpha}Q_k^{(\alpha)}(\beta t), \end{equation} 
were available, this method could have  been used to sample from the FPT rv. However, $g(t)$ is generally unknown. Therefore, 
\begin{equation} \label{eq:accrej1} 
\frac{\hat{g}_n(t)}{f_{\alpha,\beta}(t)} = p_n(t)
\end{equation}
can be used in place of \eqref{acceptance},
assuming $\hat{g}_n(t)$ as given in \eqref{expansion3} and non-negative for all $t >0.$ 
Unfortunately, since \(p_n(t)\) is a polynomial for any $n>0$, the right hand side of \eqref{eq:accrej1} is clearly unbounded on $(0, \infty)$.

We  provide a suitable modification of the standard acceptance-rejection method with the aim of sampling from the FPT rv using  \eqref{eq:accrej1}.
In the following,  suppose $T_n$ be the rv with non-negative pdf $\hat{g}_n(t)$ over $(0,\infty).$ 
The main steps of the method here proposed  can be summarized as follows:
\begin{enumerate}
\item[{\it i)}] find a constant $C$ such that $\mathbb{P}(T > C) \le \varepsilon$, for a fixed, small $\varepsilon > 0$; 
\item[{\it ii)}] for $t \leq C$ 
apply the classical acceptance-rejection method using the ratio \begin{equation}
\label{eq:accrej2}
\frac{\hat{g}_{\scriptscriptstyle{T_n|C}}(t)}{\tilde{f}_{\scriptscriptstyle{\alpha, \beta|C}}(t)} \leq M \quad \text{where} \quad 
M = \frac{{\mathbb P}(T_n \leq C)}{{\mathbb P}(X_{\alpha, \beta} \leq C)} \max_{t \in [0,C]}{p_n(t)}
\end{equation}
and  
\begin{equation} \tilde{f}_{\scriptscriptstyle{\alpha, \beta|C}}(t)= \frac{f_{\alpha,\beta}(t)}{{\mathbb P}(X_{\alpha, \beta} \leq C)} \mathbbm{1}_{(0,C]}(t),
\quad \hat{g}_{\scriptscriptstyle{T_n|C}}(t) = \frac{\hat{g}_n (t)}{{\mathbb P}(T_n \leq C)}\mathbbm{1}_{(0,C]}(t);
\label{leggeG}
\end{equation}
 \item[{\it iii)}] for $t> C$ sample from a truncated exponential rv $\bar{T}$ 
 with pdf 
\begin{equation}
g_{\scriptscriptstyle{\bar{T}}}(t)= \frac{1}{{\mathbb E}[T]}
\exp \left(-\frac{t-C}{{\mathbb E}[T]} \right) \mathbbm{1}_{(C,+\infty)}(t).
\label{leggeTn}
\end{equation}  
\end{enumerate}
The last step takes into account the FPT pdf's exponential asymptotic behavior for one-dimensional diffusion processes with steady-state distribution \cite{nobile_ricciardi_sacerdote_1986}. 
Algorithm \ref{algorithm1} outlines the proposed method for constructing an \lq\lq approximated'' sample $S$ of size $N$ from $T$.

\IncMargin{5mm}
\begin{algorithm}[]
\vspace{2mm}
\hspace{-5mm} Set the parameters $\alpha$, $\beta>0$, $n > 0$ and $\varepsilon>0$.\\[0mm]
 \SetKwBlock{Begin}{Initialise}{end}
 \Begin{
Find a constant $C$ such that $\mathbb{P}(T > C) \le \varepsilon$.\\
Set $M$ as in \eqref{eq:accrej2}.\\
$j \leftarrow 1$.\\
$S \leftarrow \{\,\}$.
}
 \SetKwBlock{Begin}{While $j < N$}{end}
 \Begin{
 \SetKwBlock{Begin}{With probability  $\varepsilon$}{end}
 \vskip 0.1cm
 \Begin{
Generate $\bar{T}$ from the truncated exponential pdf in \eqref{leggeTn}.\\
\vskip 0.1cm
$j \leftarrow j + 1$.\\
\vskip 0.1cm
$S\leftarrow S\cup T_j$.
}
 \SetKwBlock{Begin}{With probability  $1 - \varepsilon$}{end}
 \vskip 0.1cm
 \Begin{ 
 Generate $G$ from the truncated gamma pdf $\tilde{f}_{\scriptscriptstyle{\alpha, \beta|C}}$ in \eqref{leggeG}. \\ 
 \vskip 0.1cm
 Generate $U \sim U(0,1)$.\\
 \vskip 0.1cm
 \SetKwBlock{Begin}{While $U > \frac{\hat{g}_{\scriptscriptstyle{T_n|C}}(G)}{M \, \tilde{f}_{\scriptscriptstyle{\alpha, \beta|C}}(G)}$}{end} \Begin{
 \vskip 0.1cm
  Generate $G$ from the truncated gamma pdf $\tilde{f}_{\scriptscriptstyle{\alpha, \beta|C}}$ in \eqref{leggeG}. \\ 
 \vskip 0.1cm
 Generate $U \sim U(0,1)$.\\
 \vskip 0.1cm
    }
    $j \leftarrow j + 1$.\\
    \vskip 0.1cm
    $T_j \leftarrow G$.\\
    \vskip 0.1cm
    $S\leftarrow S\cup T_j$.
 }
 }
\textbf{Return} $S$.\\
 \vskip 0.1cm
 \caption{Modified acceptance-rejection method}
 \label{algorithm1}
\end{algorithm}
In addition to the user-specified input parameters, the constant $C$ in {\it i)} must be chosen for the algorithm initialization. The so-called {\it Vysochanskij-Petunin} inequality for one-sided tail bounds \cite{pukelsheim1994three} is used to achieve this goal.
\begin{thm}[Vysochanskij-Petunin inequality] If $r \ge 0$ and $X$ is a rv with unimodal density,  finite mean $\mu$ and finite variance $\sigma^2$, then
\begin{equation}\label{eq:accrecineq}
\mathbb{P}(X - \mu \ge r) \leq \begin{cases}
\frac{4}{9}\frac{\sigma^2}{\sigma^2 + r^2} &\text{if $3r^2 \ge 5\sigma^2$}, \\
\frac{4}{3}\frac{\sigma^2}{\sigma^2 + r^2} - \frac{1}{3} &\text{otherwise.}
\end{cases}
\end{equation}
\end{thm}
Since the FPT rv of diffusion processes has a unimodal pdf \cite{Uwe80}, the Vysochanskij-Petunin inequality can be applied. Setting $4 \sigma^2/[9(\sigma^2 + r^2)] = \varepsilon$ in the first inequality  \eqref{eq:accrecineq}, we recover $r=r(\varepsilon)$ as a function of $\varepsilon$ and get the condition $\varepsilon \leq 1/6$ from $3 r(\varepsilon)^2 \geq 5 \sigma^2.$ Then set $C =\mu + r(\varepsilon)$ where
\begin{equation}\label{eq:C}
r(\varepsilon) = \begin{cases}
\sqrt{\frac{4 \sigma^2}{9 \varepsilon} - \sigma^2 } &\text{if $\varepsilon \leq 1/6$}, \\
\sqrt{\frac{4 \sigma^2}{1+ 3 \varepsilon} - \sigma^2 } &\text{if $ 1/6 < \varepsilon \leq 1$}.
\end{cases}
\end{equation}
\begin{rem}\label{rem:accrej1}
From \eqref{eq:C},
$C$ may become arbitrarily large by decreasing $\varepsilon$. However, as $C$ increases, so does $M$, increasing the chance of rejection and subsequently the required number of iterations.
\end{rem}
The quality of the outcome of Algorithm \ref{algorithm1} relies on the approximation $\hat{g}_n(t)$ and the selection of an exponential distribution for $t > C$. 

A theoretical justification for Algorithm \ref{algorithm1} is provided in the following. We first calculate the cdf of the rv $Y,$ whose observations are generated by Algorithm \ref{algorithm1}, and then  prove that such a distribution turns out to be a good approximation of the FPT cdf.
\begin{lem}\label{lem:accrej1}
If $Y$ denotes the rv sampled at the end of each cycle of Algorithm \ref{algorithm1}, then 
\begin{equation*} 
        \mathbb{P}(Y \leq t ) = \varepsilon 
        \left[1-\exp\left(-\frac{t-C}{\mathbb E(T)} \right) \right] 
        + (1-\varepsilon)\mathbb{P}(T_n \leq t \;|\; T_n \leq C),\qquad t > 0
    \end{equation*}
where $C, n$ and $\varepsilon$ are given in Algorithm \ref{algorithm1}, and $T_n$ is the rv with pdf $\hat{g}_n(t).$ 
\end{lem}
\begin{proof}
According to Algorithm \ref{algorithm1}, we have
\begin{equation}
{\mathbb P}(Y \leq t) = {\mathbb P}(X = 1) 
{\mathbb P} (\tilde{T} \leq t) + {\mathbb P}(X = 0)  {\mathbb P}(G \leq t \; |\; G \, {\rm accepted})
\label{lawalternative}
\end{equation}
where $X$ is a Bernoulli rv of parameter $\varepsilon \in (0,1)$ independent from the rv
$\tilde{T},$ with truncated exponential pdf $g{\scriptscriptstyle{\tilde{T}}}(t)$ in \eqref{leggeTn}, and the rv $G$ with truncated gamma pdf $\tilde{f}_{\scriptscriptstyle{\alpha, \beta|C}}(t)$ in \eqref{leggeG}. Thus from \eqref{leggeTn} and \eqref{lawalternative}, we have 
\begin{equation}
{\mathbb P}(Y \leq t) = \varepsilon  \left[1-\exp\left(-\frac{t-C}{\mathbb E(T)} \right) \right] + (1 - \varepsilon) {\mathbb P}(G \leq t \; |\; G \, {\rm accepted}).
\label{lawalternative1}
\end{equation}
For the latter term in \eqref{lawalternative1} observe that 
\begin{equation}
{\mathbb P}(G \leq t \;| \; G\text{ accepted})  = M \, \mathbb{P} \textbf{}(G \leq t, G\text{ accepted})
\label{cond0}
\end{equation}
since ${\mathbb P} (G\text{ accepted}) = 1/M$ with $M$ given in \eqref{eq:accrej2}. Moreover
\begin{align}
\mathbb{P} \textbf{}(G \leq t, \, G\text{ accepted}) & = \int_0^{\infty} 
\mathbb{P} \textbf{}(G \leq t, \, G\text{ accepted} \; | \; G=x) \, \tilde{f}_{\scriptscriptstyle{\alpha, \beta|C}}(x) \,\, {\rm d}{x} \nonumber \\
& = \int_0^{\infty} 
\mathbb{P}(G \leq t \; | \; G=x) \, 
\mathbb{P} (G\text{ accepted} \; | \; G=x) \, \tilde{f}_{\scriptscriptstyle{\alpha, \beta|C}}(x) \,\, {\rm d}{x}  
\label{cond1}
\end{align}
since $(G\leq t)$ and $(G \,\, {\rm accepted})$ are conditionally independent events.
Observe that
$\mathbb{P}(G \leq t \; | \; G=x) = \mathds{1}_{x \leq t}$ and
\begin{align}
\mathbb{P}(G\text{ accepted} \; | \; G=x) & =  \mathbb{P} \left( U \leq \frac{\hat{g}_{\scriptscriptstyle{T_n|C}}(G)}{M \,  \tilde{f}_{\scriptscriptstyle{\alpha, \beta|C}}(G)} \; \bigg| \; G=x \right) \nonumber \\
& = \mathbb{P} \left( U \leq \frac{\hat{g}_{\scriptscriptstyle{T_n|C}}(x)}{M \, \tilde{f}_{\scriptscriptstyle{\alpha, \beta|C}}(x)} \right)=\frac{\hat{g}_{\scriptscriptstyle{T_n|C}}(x)}{M \,  \tilde{f}_{\scriptscriptstyle{\alpha, \beta|C}}(x)} \label{cond2}
\end{align}
since $U$ is a rv with uniform distribution over $(0,1).$
Plugging \eqref{cond2} in \eqref{cond1} and the resulting integral in \eqref{cond0}, we get
$${\mathbb P}(G \leq t \;| \; G\text{ accepted})  = M\int_{0}^{\infty}\mathds{1}_{x \leq t}\frac{\hat{g}_{\scriptscriptstyle{T_n|C}}(x)}{M \, \tilde{f}_{\scriptscriptstyle{\alpha, \beta|C}}(x)} \, \tilde{f}_{\scriptscriptstyle{\alpha, \beta|C}}(x)\, {\rm d} x = \int_{0}^{t}\hat{g}_{\scriptscriptstyle{T_n|C}}(x) \, {\rm d} x, $$
and the result follows from \eqref{leggeG}.
\end{proof}
The following is a technical lemma necessary for the subsequent proposition. The lemma gives an upper bound of the error in approximating the FPT pdf $g(t)$ with the Laguerre-Gamma expansion $\hat{g}_n(t)$ in \eqref{expansion3}
for $t \leq C.$
\begin{lem} 
\label{lemma_dist} Under the same hypotheses as Lemma \ref{lem:accrej1}, we have
$$\Big|\mathbb{P}(T_n \leq t \;|\; T_n \leq C) - \mathbb{P}(T \leq t \;|\; T \leq C)\Big|\leq\frac{1}{\mathbb P(T_{n} \leq C) }\left[\sum_{k\geq n+1}(a_{k}^{(\alpha)})^{2}\right]^{\frac12} \!\!\!\! \!\!\!\! \left(1+\frac1{\mathbb P(T\leq C)}\right)$$
with $a_k^{(\alpha)}$ given in \eqref{approximationgen2bis}.
\end{lem}
\begin{proof}
If $ 
g_{\scriptscriptstyle{T|C}}(t) =  \mathbbm{1}_{(0,C]}(t) g(t) / {\mathbb P}(T \leq C)$  is the truncated FPT pdf for $t \leq C,$ 
and $\hat{g}_{\scriptscriptstyle{T_n|C}}(t)$ is given in \eqref{leggeG}, we have
$$
\begin{aligned}
\Big| \hat{g}_{\scriptscriptstyle{T_n|C}}(t) -g_{\scriptscriptstyle{T|C}}(t) \Big| & =\left|\frac{\hat{g}_n(t)}{\mathbb P(T_{n}\leq C)}-\frac{g(t)}{\mathbb P(T\leq C)}\right|  \\
&\leq \frac{1}{\mathbb P(T_n\leq C)}\left|\hat{g}_n(t) - g(t)\right| + g(t) \left| \frac{1}{\mathbb P(T_n\leq C)} -\frac{1}{\mathbb P(T\leq C)}\right|.
\end{aligned}
$$
Using the previous inequality, the difference between the truncated FPT cdf and its approximation by means of the Laguerre-Gamma expansion may be bounded as follows:
    \begin{eqnarray}
    &&|\mathbb{P}(T_n \leq t \;|\; T_n \leq C) - \mathbb{P}(T \leq t \;|\; T \leq C)|
    =\left|\int_{0}^{t} [ \hat{g}_{\scriptscriptstyle{T_n|C}}
    (s) - g_{\scriptscriptstyle{T|C}}(s)] \, {\rm d}s \right| \nonumber  \\
&&\leq \int_{0}^{\infty}|
\hat{g}_{\scriptscriptstyle{T_n|C}}
    (t) - g_{\scriptscriptstyle{T|C}}(t)] \, | {\rm d}t \leq \frac{1}{\mathbb P(T_{n}\leq C)}\int_{0}^{\infty}\frac{|\hat{g}_{n}(s)-g(s)|}{\varphi_{\alpha,\beta}(s)}\varphi_{\alpha,\beta}(s)\, {\rm d} s \nonumber  \\
&&+\left|\frac{1}{\mathbb P(T_{n}\leq C)}-\frac{1}{\mathbb P(T\leq C)}\right|\underbrace{\int_{0}^{\infty}g(s) \, {\rm d} s}_{=1} \nonumber \\
&&=\frac{1}{\mathbb P(T_{n}\leq C)}\left\|\frac{\hat{g}_n-g}{f_{\alpha,\beta}}\right\|_{{\mathcal L}^1(\nu)} +\left|\frac{1}{\mathbb P(T_{n}\leq C)}-\frac{1}{\mathbb P(T\leq C)}\right| \label{bound}
\end{eqnarray}
where ${\mathcal L}^1(\nu)$ is the Hilbert space of the integrable functions with respect to the measure $\nu$ having density $f_{\alpha,\beta}(t).$
Observe that 
\begin{equation}
\left\|\frac{\hat{g}_n-g}{f_{\alpha,\beta}}\right\|_{{\mathcal L}^1(\nu)}
\leq \left\|\frac{\hat{g}_n-g}{f_{\alpha,\beta}}\right\|_{{\mathcal L}^2(\nu)} 
=\left[\sum_{k\geq n+1} (a_k^{(\alpha)})^2\right]^{1/2}
\label{primooaddendo}
\end{equation}
where the last equality follows from \eqref{errorPars}.
Finally, since
\begin{equation}
\left|\frac{1}{\mathbb P(T_{n}\leq C)}-\frac{1}{\mathbb P(T\leq C)}\right| \leq \frac{1}{\mathbb P(T_{n}\leq C)} \left( 1 + \frac{\mathbb P(T_{n}\leq C)}{\mathbb P(T\leq C)} \right) 
\label{secondoaddendo}
\end{equation}
with ${\mathbb P}(T_{n}\leq C) \leq 1,$
plugging \eqref{secondoaddendo} and \eqref{primooaddendo} in \eqref{bound}
the result follows. 
\end{proof}
The following proposition gives the  theoretical justification of the proposed acceptance-rejection Algorithm \ref{algorithm1}.
\begin{prop}\label{prop:accrej2}
 Under the same hypotheses as Lemma \ref{lem:accrej1}, 
for every $\delta > 0 $, there exist a finite constant $C_\delta$ and an integer $n_\delta\in \mathbb{N}$  such that, for  $n>n_\delta$ and $C>C_\delta$ in Algorithm \ref{algorithm1}, we have
$$|\mathbb{P}(Y \leq t) - \mathbb{P}(T\leq t)| < \delta, \qquad 
 t > 0.$$
    \end{prop}
\begin{proof}
Fix $\delta$, $n > 0$ and $\varepsilon > 0$ and let $C>0$ be the corresponding quantity calculated in Algorithm \ref{algorithm1}. 
By plugging ${\mathbb P}(T \leq t) = {\mathbb P}(T \leq C) {\mathbb P}(T \leq t | T \leq C) + {\mathbb P}(T > C){\mathbb P}(T \leq t| T > C)$ in $|\mathbb{P}(Y \leq t) - \mathbb{P}(T\leq t)|$ and using Lemma \ref{lem:accrej1}, the following bound can be recovered: 
\begin{equation*}
|{\mathbb P}(Y \leq t) - {\mathbb P}(T \leq t)| \leq P_1 + P_2 
\end{equation*}
where 
\begin{eqnarray*}
P_1 & := & |\varepsilon \mathbb{P}(\bar{T} \leq t \;|\; \bar{T}  > C) - \mathbb{P}(T > C) \mathbb{P}(T \leq t\;|\; T > C)| \\ 
P_2 & := & |(1-\varepsilon)\mathbb{P}(T_n \leq t \;|\; T_n \leq C)  - \mathbb{P}(T \leq C)\mathbb{P}(T \leq t \;|\; T \leq C)|,
\end{eqnarray*}
$\bar{T}$ is the truncated exponential rv with pdf in \eqref{leggeTn}
and 
$T_n$ is the rv with  pdf $\hat{g}_n(t).$  
Now, let us bound $P_1$. 
Adding and subtracting the quantity $\varepsilon \mathbb{P}(T \leq t\;|\; T > C)$ in $P_1$,  we get
$$
P_1 \leq P_{1,1} + P_{1,2} \quad \text{with} \quad 
\left\{\begin{array}{lll}
P_{1,1} & := &  \varepsilon |\mathbb{P}(\bar{T}  \leq t \;|\; \bar{T}  > C) - \mathbb{P}(T \leq t\;|\; T > C)|, \\
P_{1,2} & :=  & \mathbb{P}(T \leq t\;|\; T > C) |\varepsilon - \mathbb{P}(T > C)|.
\end{array}\right.
$$
Since ${\mathbb P}(T>C) \leq \varepsilon,$ thanks to Remark \ref{rem:accrej1} and to the exponential behaviour of the tails of the FPT pdf, it is always possible to find $C_{\delta, 1}$ (big enough) such that $P_{1, 1} < \frac{\delta}{4}$ and $P_{1, 2} < \frac{\delta}{4}$.
Let us apply the same strategy  to $P_2$.
Adding and subtracting the quantity $(1-\varepsilon) \mathbb{P}(T \leq t\;|\; T \leq C)$ in $P_1$,  we get
$$
P_2 \leq P_{2,1} + P_{2,2} \quad \text{with} \quad 
\left\{\begin{array}{lll}
P_{2,1} & := &  (1-\varepsilon) |\mathbb{P}(T_n \leq t \;|\; T_n \leq C) - \mathbb{P}(T \leq t \;|\; T \leq C)|, \\
P_{2,2} & :=  & \mathbb{P}(T \leq t \;|\; T \leq C) |\varepsilon -\mathbb{P}(T > C)|.
\end{array}\right.
$$
From Lemma \ref{lemma_dist}, the quantity $P_{2, 1}$ can be made sufficiently small for a suitable $n,$
since the remainder 
\eqref{errorPars} goes to zero as $n$ increases. Therefore, 
thanks to Remark \ref{rem:accrej1} and the exponential behaviour of the FPT pdf tail, we can always find $C_{\delta, 2}$ and $n_\delta$ such that $P_{2, 1} < \frac{\delta}{4}$ and $P_{2, 2} < \frac{\delta}{4}$. Setting $C_\delta = \max(C_{\delta, 1},\; C_{\delta, 2})$  concludes the proof.
\end{proof}
\begin{rem} 
    Note that although Proposition \ref{prop:accrej2} guarantees that we can always choose a constant $C$ and an order $n$ such that the cdf of $Y$ sampled in each cycle of Algorithm \ref{algorithm1} is sufficiently close to the cdf of $T$, as $C$ increases, $M = \max_{t \in [0,C]}{p_n(t)}$ will also increase or at most remain constant. However, the probability of acceptance in Algorithm \ref{algorithm1} decreases if $M$ increases, as we have already noted in Remark \ref{rem:accrej1}. Therefore, there is a trade-off between increasing $C,$ to achieve better accuracy, and the running time of the proposed algorithm.
\end{rem}
The ability of Algorithm \ref{algorithm1} to generate a satisfactory \lq\lq approximated" sample from the FPT rv is shown in Fig.~ \ref{acc_rej}. In the latter we consider case A with the standardization procedure outlined in Section \ref{choiceAB}. Hence in \eqref{leggeG}, instead of $\hat{g}_{n}(t)$, we use the approximation $\tilde{g}_{n}(t)$ of the pdf $\tilde{g}(t)$ corresponding to
$
\tilde{T}=T/\sigma_T,
$
where $\sigma_T$ is the standard deviation of $T$.
\begin{figure}[h]
\centering
\includegraphics[scale = 0.6]{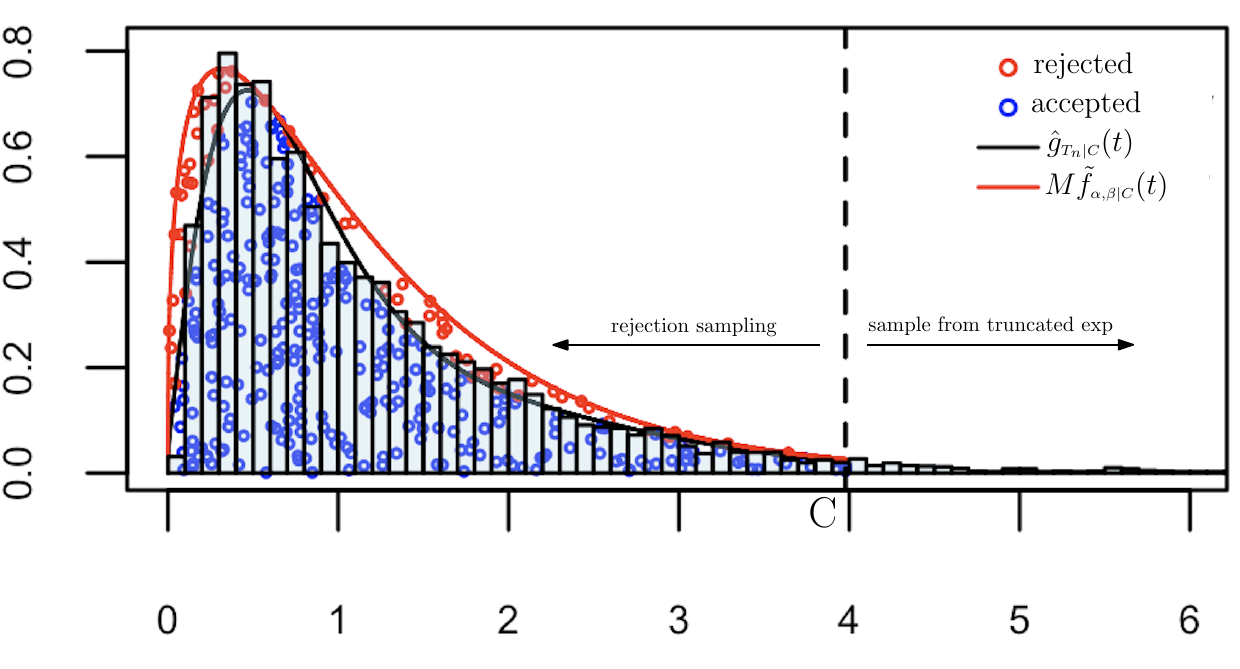}
\caption{Plot of rejected  draws (in red) and accepted  draws (in blue) for the acceptance-rejection part of Algorithm \ref{algorithm1} with $N = 300$, $\epsilon = 0.05$ and $n = 10$ applied to case A, together with the corresponding $M\tilde{f}_{\scriptscriptstyle{\alpha, \beta|C}}(t)$ and $\tilde{g}_{\scriptscriptstyle{T_n|C}}(t) = \frac{\tilde{g}_n (t)}{{\mathbb P}(T_n \leq C)}\mathbbm{1}_{(0,C]}(t)$ as in \eqref{leggeG}, with $C = 3.97, \, \alpha = 0.367$ and $\beta = 1.17$. On the same plot an histogram computed on the sample of size $N = 300$ arising from Algorithm \ref{algorithm1} with the aforementioned parameters is shown.
}\label{acc_rej}
\end{figure}

\section{Conclusions}

In this paper we develop a method 
for approximating the FPT pdf and cdf of a CIR process that relies on a series expansion involving the generalized Laguerre polynomials and the gamma pdf. The significant improvements of the method proposed recently in \cite{di2021cumulant} follow two main directions. The first one is theoretical: we detail a study of the approximation error along with considerations on the choice of the reference pdf parameters as well as some sufficient conditions  for the non-negativity of the approximant in a right hand side of the origin and in the tail. The second direction is more of a numerical nature. We propose computational strategies to overcome some numerical issues like the possible negativity of the resulting function, a feature that is undesirable in the approximation of a pdf. Methods of standardization according to considerations on dispersion measures, stopping criteria and iterative mechanisms for the choice of the best order of truncation of the series are investigated and applied to three main classes of target pdfs here considered.
Moreover, we investigate the performance of the approximation of the FPT cdf that constitutes a novelty in this framework. The developed iterative procedure is of general nature and can be applied
both starting from the model or from FPT sample data. The method, in fact, works well both using theoretical cumulants or their values estimated from the data. The discussion can be easily  extended to the case of other diffusion processes as long as it is possible to calculate their FPT moments and the reference pdfs are moment determined   \cite{stoyanov2020new}.

As a side result an acceptance-rejection-like method, that makes use of the approximation, is proposed. It  allows the generation of FPT data, although its distribution is
unknown. Its usefulness is increased by the lack of existing exact methods for simulating the sample paths of the CIR process and consequently its FPTs. Furthermore, note that the suggested acceptance-rejection like procedure could be directly extended to a wider class of pdfs (not necessarily of FPT rv) which, although unknown, may admit a series expansion of the type discussed in this paper.  

\section{Appendix}
The logarithmic (partition) polynomials $\left\{P_{k}\right\}$ are \cite{Charalambides}
$$
P_{k}\left(x_{1}, \ldots, x_{k}\right)=\sum_{j=1}^{k}(-1)^{j-1}(j-1) ! B_{k, j}\left(x_{1}, \ldots, x_{k-j+1}\right)
$$
where $\left\{B_{k, j}\right\}$ are the partial exponential Bell polynomials. For a fixed positive integer $k$ and $j=1, \ldots, k$, the $j$-th partial exponential Bell polynomial in the variables $x_{1}, x_{2}, \ldots, x_{k-j+1}$ is an homogeneous polynomial of degree $j$ given by
\begin{equation}
B_{k, j}\left(x_{1}, \ldots, x_{k-j+1}\right)=\sum \frac{k !}{\lambda_{1} ! \lambda_{2} ! \cdots \lambda_{k-j+1} !} \prod_{i=1}^{k-j+1}\left(\frac{x_{i}}{i !}\right)^{\lambda_{i}}
\label{partialBell}
\end{equation}
where the sum is taken over all sequences $\lambda_{1}, \lambda_{2}, \ldots, \lambda_{k-j+1}$ of non-negative integers such that
$$
\lambda_{1}+2 \lambda_{2}+\cdots+(k-j+1) \lambda_{k-j+1}=k, \quad \lambda_{1}+\lambda_{2}+\cdots+\lambda_{k-j+1}=j.
$$
The complete Bell (exponential) polynomials $\left\{Y_{k}\right\}$ are \cite{Charalambides}
$$
Y_{k}\left( x_{1}, \ldots, x_{k}\right) = \sum_{j=1}^{k} B_{k, j}\left(x_{1}, \ldots, x_{k-j+1}\right), \,\,\, k \geq 1
$$
where $\left\{B_{k, j}\right\}$ are the partial exponential Bell polynomials (\ref{partialBell}) and $Y_0=1.$ 

\section{Acknowledgements}
The authors would like to thank Andrea Marafante for his decisive contribution in developing and implementing the code necessary for the contents of this paper.
A special thank goes to Antonio Di Crescenzo for some inspiring discussions. 
\par
{\it Funding:} the research of E.D. and G.D. was partially supported by the MIUR-PRIN
2022 project \lq\lq Non-Markovian dynamics and non-local equations\rq \rq, 202277N5H9. The authors participate in the INdAM - GNAMPA Project, CUP E53C23001670001.

\bibliographystyle{abbrv}
\bibliography{refs}

\end{document}